\documentclass[opre,nonblindrev]{informs3}
\OneAndAHalfSpacedXI

\usepackage{endnotes}
\let\footnote=\endnote
\let\enotesize=\normalsize
\def\notesname{Endnotes}%
\def\makeenmark{$^{\theenmark}$}
\def\enoteformat{\rightskip0pt\leftskip0pt\parindent=1.75em
  \leavevmode\llap{\theenmark.\enskip}}
  
  \usepackage{natbib}
 \bibpunct[, ]{(}{)}{,}{a}{}{,}%
 \def\bibfont{\small}%
\TheoremsNumberedThrough     
\ECRepeatTheorems

\EquationsNumberedThrough    

\usepackage{custom_formats}
\usepackage{custom_rafid}

\theoremstyle{TH}

\newcommand{\iz}[1]{\textbf{\textcolor{blue!67}{#1}}}

\usetikzlibrary{arrows,automata}
\usetikzlibrary{shapes,arrows,shadows,positioning,decorations.pathreplacing,trees,calc,patterns, fit}

\tikzstyle{conceptarrow} = [dashed, line width=0.25mm,<->, ]
\tikzstyle{arrow} = [line width=0.25mm,->, ]
\tikzstyle{arrowlight} = [line width=0.25mm,>=stealth]

\tikzstyle{noBox} = [rectangle, rounded corners, minimum width=1.cm, minimum height=1.5cm, align=center, text centered] 
\tikzstyle{noBoxWide} = [rectangle, rounded corners, minimum width=2cm, minimum height=1.5cm,text width=2cm, text centered]
\tikzstyle{process} = [rectangle, rounded corners, minimum width=2.25cm, minimum height=1.cm, align=center, text centered, font=\bf\sffamily]
\tikzstyle{box} = [rectangle, ultra thick, rounded corners, minimum width=2cm, minimum height=1.cm, text width=2cm, text centered, draw=black,fill=black!15!white]
\tikzstyle{dashedBorder} = [draw, rounded corners = 14pt, dashed, thick]

\makeatletter
\newcommand{\vast}{\bBigg@{4}}
\newcommand{\Vast}{\bBigg@{5}}
\makeatother

\newcommand{\prob}[1]{\mathbf{#1}}

\newcommand{\FOP}{\prob{FOP}}
\newcommand{\FOPL}{\prob{FOP}\text{--}\prob{L}}
\newcommand{\FOPC}{\prob{FOP}\text{--}\prob{CON}}
\newcommand{\FOPCSF}{\prob{FOP}\text{--}\prob{CON}\text{--}\prob{SF}}
\newcommand{\FOPCVX}{\prob{FOP}\text{--}\prob{CVX}}
\newcommand{\FOPMDP}{\prob{FOP}\text{--}\prob{MDP}}
\newcommand{\FOPMIP}{\prob{FOP}\text{--}\prob{MI}}

\newcommand{\IOPp}{\prob{IOP}\text{--}\prob{C}}
\newcommand{\IOPd}{\prob{IOP}\text{--}\prob{DD}}
\newcommand{\IOPdr}{\prob{IOP}\text{--}\prob{DD}\text{--}\prob{R}}
\newcommand{\IOPddro}{\prob{IOP}\text{--}\prob{DD}\text{--}\prob{DRO}}

\newcommand{\dataset}{\set{D}}
\newcommand{\empiricaldist}{\field{P}_N}
\newcommand{\Xfeas}{\set{X}}
\newcommand{\Xopt}{\set{X}^{\mathrm{opt}}}
\newcommand{\Thetainv}{\boldsymbol{\Theta}^{\mathrm{inv}}}

\begin{document}

\RUNAUTHOR{Chan, Mahmood, Zhu}
\RUNTITLE{A Review of Inverse Optimization}

\TITLE{Inverse Optimization: Theory and Applications}

\ARTICLEAUTHORS{%
\AUTHOR{Timothy C. Y. Chan}
\AFF{Department of Mechanical and Industrial Engineering, University of Toronto, ON, Canada \\ \EMAIL{tcychan@mie.utoronto.ca}}
\AUTHOR{Rafid Mahmood}
\AFF{NVIDIA Corporation, \EMAIL{rafid.mahmood@mail.utoronto.ca}}
\AUTHOR{Ian Yihang Zhu}
\AFF{Department of Mechanical and Industrial Engineering, University of Toronto, ON, Canada \\ \EMAIL{i.zhu@mail.utoronto.ca}}
}


\ABSTRACT{Inverse optimization describes a process that is the ``reverse" of traditional mathematical optimization. Unlike traditional optimization, which seeks to compute optimal decisions given an objective and constraints, inverse optimization takes decisions as input and determines an objective and/or constraints that render these decisions approximately or exactly optimal. In recent years, there has been an explosion of interest in the mathematics and applications of inverse optimization. This paper provides a comprehensive review of both the methodological and application-oriented literature in inverse optimization. }


\maketitle



\vspace{-0.4cm}

\section{Introduction}

In traditional mathematical optimization, one takes as input an objective and a set of constraints to generate an optimal decision. In inverse optimization, decisions are given as input and an objective and/or constraints is the output. 
Specifically, the goal of inverse optimization is to determine parameters of an optimization model---the ``forward'' model---that render a set of decisions approximately or exactly optimal with respect to this forward model. Solving for this set of parameters is itself an optimization model, known as the ``inverse'' model.

In recent years, there has been an explosion of interest in the mathematics and applications of inverse optimization. For example, it has found diverse applications in areas such as transportation, healthcare, and power systems, where we can use an observed data set of decisions to estimate a decision-making model that best reproduces these observations. Decision data can include individual routing choices, electricity consumption patterns or medical treatments, and the inverse optimization model can be used to estimate route-choice preferences, utility functions or clinical treatment objectives that are most consistent with the observed decisions. In these examples, inverse optimization provides a mathematical framework for estimating latent parameters and subjective preferences within decision-making problems.

The tractability of an inverse optimization problem depends on the complexity of the forward model and the desired properties sought in the inverse model. Different applications require different modeling assumptions, leading to many different inverse models and corresponding solution methods. Nonetheless, all models can broadly be characterized as a combination of elements along the following three major dimensions (see Figure~\ref{fig:IO_taxonomy}):

\begin{itemize}
    \item \textbf{Forward problem structure:} 
    This dimension describes the structure of the forward model. Possibilities include linear, conic, convex, or discrete optimization models. Sequential decision-making models (e.g., Markov decision processes) are also possible.

    \item \textbf{Parameter type:} 
    The parameters to be estimated may reside in the objective, constraints, or both. The structure and tractability of the inverse optimization problem depend heavily on which parameters are estimated and the forward problem structure.

    \item \textbf{Model-data fit:} 
    Depending on the application, it may be necessary to achieve a ``perfect fit'' between the estimated parameters of the forward model and the decisions, i.e., finding parameters such that the decisions are optimal. In problems where this is unnecessary or impossible, the inverse problem may instead aim to maximize a suitable measure of fitness.

\end{itemize}

Since the treatment of the first two dimensions is relatively similar across the third dimension, we specifically assign the labels of \emph{Classical Inverse Optimization} and \emph{Data-Driven Inverse Optimization} to distinguish between the models that assume perfect model-data fit and those that do not. Classical problems, which dominate the early literature on inverse optimization, are often used to introduce new reformulation techniques of the inverse problem. As we will discuss in later sections, the models are also relevant in applications where it is necessary to find parameters in the forward model that render the decisions optimal. On the other hand, data-driven models draw upon the techniques and methods employed in classical models but also consider an additional layer of complexity arising from the possibility of imperfect model-data fit. Such models are relevant when there is model mis-specification or noisy observations in the decision data. This literature is distinguished by loss functions that correspond to various sub-optimality measures.

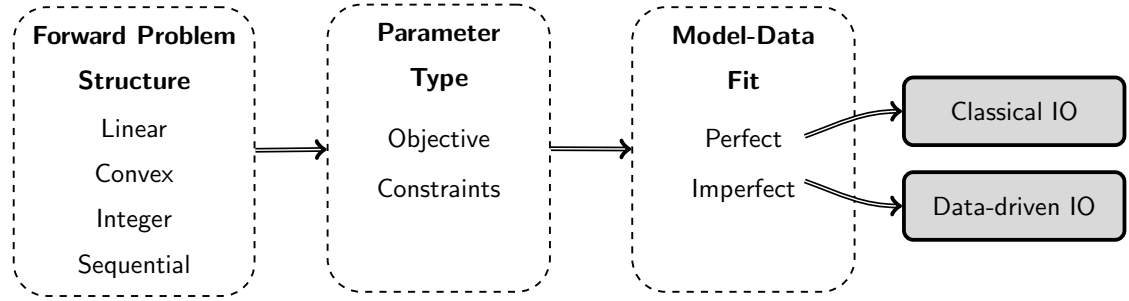
\begin{figure}[t]
    \centering
    \scalebox{0.9}{%
    \tikzset{every picture/.style={line width=0.75pt}} 

\begin{tikzpicture}[font=\sffamily]

    \node(FPS) [process, minimum width=3cm] {Forward Problem \\ Structure};
    \node(FPSprob) [noBox, below of = FPS, yshift=-0.5cm] {\\[2em] Linear \\ Convex \\ Integer \\ Sequential};
    
    \node(PT) [process, right of = FPS, xshift=3.5cm, minimum width=3cm] {Parameter \\ Type};
    \node(PTprob) [noBox, below of = PT, yshift=-0.5cm] {\\[2.2em] Objective \\[0.2em] Constraints \\[1.2em] ~};
    
    \node(Ass) [process, right of = PT, xshift=3.5cm, minimum width=3cm] {Model-Data \\ Fit};
    \node(Assprob) [noBox, below of = Ass, yshift=-0.5cm] {\\[2.2em] Perfect \\[0.2em] Imperfect \\[1.2em] ~ };
    
    \node [dashedBorder, fit=(FPS)(FPSprob)] (FPSborder) {};
    \node [dashedBorder, fit=(PT)(PTprob)] (PTborder) {};
    \node [dashedBorder, fit=(Ass)(Assprob)] (Assborder) {};

    \draw[arrow, double, ->] (FPSborder) to (PTborder);
    \draw[arrow, double, ->] (PTborder) to (Assborder);
    
    \node[box, minimum width=3cm, text width=3cm, right of = Assprob, xshift=3cm, yshift=0.7cm] (CIO) {Classical IO};
    \node[box, minimum width=3cm, text width=3cm, right of = Assprob, xshift=3cm, yshift=-0.7cm] (DDIO) {Data-driven IO};
        
    \draw[arrow, double, out=20, in=180] (Assprob)  to (CIO);
    \draw[arrow, double, out=-20, in=180] (Assprob) to (DDIO);

\end{tikzpicture}
    }
    \caption{A taxonomy of inverse optimization models. The arrows correspond to the sequence of steps used to build an inverse model. Specifically, a forward model is proposed, a subset of parameters are identified, and the properties that the parameters must satisfy are selected.} \label{fig:IO_taxonomy}
\end{figure}

In this paper, we provide (i) a systematic framework for describing and categorizing inverse optimization models, and (ii) a comprehensive review of both methodological and application-oriented literature. We first review different techniques, solution methods, and model properties of all existing inverse models using the above taxonomy. We then consolidate the extensive list of application areas in which inverse optimization is increasingly applied. To our knowledge, there is only one prior survey on inverse optimization \citep{heuberger2004inverse}. That paper provides a theoretical review of select classical inverse models over particular 0-1 combinatorial forward problems, which was the focus of much early research in inverse optimization. Since then, we have observed a breadth of new models, methods, and applications, many specific to the data-driven literature. Our review thus covers techniques applicable to broad classes of problem structures, considers new developments specific to data-driven models, and examines the extensive list of modern applications.

Finally, we briefly remark that inverse optimization falls under the umbrella of ``inverse problems'', in which observational data on the ``effects'' of a process are used to estimate unobservable ``causes'' generating outcomes~\citep{tarantola2005inverse, kaipio2006statistical}. 
Inverse optimization is distinguished from other inverse problems by the fact that the process is a decision-making problem modeled via mathematical programming, while the data consists of solutions to these problems; thus, the methods and applications of inverse optimization differ from the broader literature.


\subsection{Notation}
Vectors and matrices are denoted in bold and sets in calligraphic. 
Let $\field{R}^n$, $\field{R}^n_+$, and $\field{Z}^n$ be the set of real-valued vectors, non-negative vectors, and integer vectors of dimension $n$, respectively. Similarly, $\field{S}^{n \times n}_+$ denotes the set of positive semi-definite matrices of size $n\times n$. 
$\Ind\{\cdot\}$ denotes the indicator function. Let $\interior(\mC)$ be the interior of a set $\mC$. Similarly, let $\extreme(\mC)$ denote the set of extreme points of the convex hull of $\mC$. A unit probability mass at $\bhx$ is represented by $\delta_{\bhx}$.
We denote probability distributions by $\field{P}$ or $\field{Q}$. 
Given a set $\{\bhx_i\}_{i=1}^N$ of $N$ points, we denote the empirical distribution as $\field{P}_N := \sum_{i=1}^N \delta_{\bhx_i}$. 
Finally, $\| \cdot \|_p$ denotes a general $p$-norm for $p \geq 1$. 

\section{Fundamentals of Inverse Optimization: A Road Map of This Paper}
\label{sec:problem_def}

Inverse optimization modeling begins by defining a parametric forward optimization model that represents the decision-generating process of one or more agents. Then, given a set of target or observed decisions, the inverse optimization model determines parameters of this forward model that render the decisions approximately or exactly optimal. Figure~\ref{fig:io_pipeline} illustrates these relationships.


\begin{figure}[t]
\centering
\scalebox{0.9}{%
\begin{tikzpicture}[font=\sffamily]
        \node[cloud, cloud puffs=15.7, cloud ignores aspect, minimum width=4cm, minimum height=2cm, align=center, draw] (fwd) at (0, 0) {Decision-generating \\ process(es)};

        \draw (-3,0)    node[anchor=east, minimum width=2cm, text width=1cm, align=center] (agentIn) {Agent(s)};
        \draw [arrow] (agentIn) -- (fwd);
        
        \draw (-3,-3.5)   node[anchor=east, text width=2cm, align=center] (modelOut) {Estimated \\ parameters};
        \draw (0,-3.5)    node[box, minimum width=4cm, text width=4cm, minimum height=1.25cm] (inv) {Inverse optimization \\ model};
        \draw (0,-2)    node[minimum width=4cm, align=center, text width=6cm] (fwdmodel) {Parametric forward model};

        \draw (3,-2)    node[anchor=west, text width=2cm, align=center] (obsIn) {Decisions};

        \draw [arrow, dashed] (fwd) -- (fwdmodel);
        \draw [arrow] (fwdmodel) -- (inv);
        \draw [arrow] (inv) -- (modelOut);

        \draw [arrow] (obsIn) |- (inv);
        \draw [arrow] (fwd) -| (obsIn);

\end{tikzpicture}
}
    \caption{In an inverse optimization problem, agents generate decisions that the inverse optimizer observes. The inverse optimizer proposes a forward model of agent behavior and estimates the parameters under which the forward model best supports the decisions.} 
    \label{fig:io_pipeline}
\end{figure}
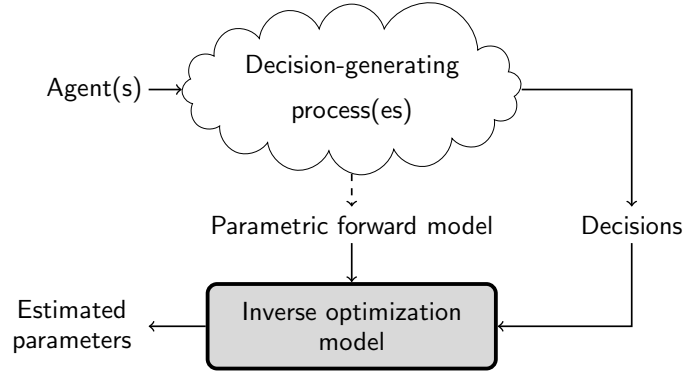

\subsection{The forward problem}
\label{subsec:problem_def_forward}

Mathematically, let the forward optimization model be 
\begin{align*}
    \FOP(\btheta) := \min_{\bx} \;\left\{ f(\bx, \btheta) \;|\; \bx \in \Xfeas(\btheta) \right\}. 
\end{align*}
This model is characterized by a parameter $\btheta$ from a parameter class $\bTheta$ that controls the objective $f(\bx, \btheta)$, feasible set $\Xfeas(\btheta)$, or both. The objective $f(\bx, \btheta)$ is typically assumed to be a convex and differentiable parametric function in both $\bx$ and $\btheta$. 
Unless otherwise stated, we assume that $\bTheta$ is a polyhedron.
For problems where $\btheta$ parameterizes only the objective function, i.e., $\mX(\btheta) = \set{X}$, common choices for modeling the objective include:
\begin{itemize}
    \item \textbf{Linear:} $f(\bx, \btheta) = \btheta^\tpose \bx$, where $\bTheta \subseteq \field{R}^n$. 
    
    \item \textbf{Quadratic:} $f(\bx, \btheta := (\bPhi, \bpsi)) = \bx^\tpose \bPhi \bx + \bpsi^\tpose \bx$, where $\bTheta \subseteq \{(\bPhi, \bpsi) \;|\; \bPhi \in \field{S}^{n\times n}_{+} ,\; \bpsi \in \field{R}^n \}$. 
    
    \item \textbf{Convex-separable bases:} $f(\bx, \btheta) = \sum_{b=1}^B \theta_b f^{(b)}\left(\bx\right)$, where $f^{(1)}(\bx), f^{(2)}(\bx), \dots, f^{(B)}(\bx)$ are $B$ convex basis functions and $\bTheta \subseteq \field{R}^{B}$.
\end{itemize}

The feasible set is assumed to have the form $\Xfeas(\btheta) := \{ \bx \in \field{R}^q \times \field{Z}^{n-q} \;|\; \bg(\bx, \btheta) \leq \bzero \}$, where $q \in \{0, 1, \ldots, n\}$ and $\bg(\bx, \btheta) = (g_1(\bx, \btheta), \cdots, g_m(\bx, \btheta))$ is composed of $m$ convex functions. 
Finally, let 
\begin{align*}
    \Xopt(\btheta) := \argmin_{\bx} \;\left\{ f(\bx, \btheta) \;|\; \bx \in \Xfeas(\btheta) \right\},
\end{align*}
denote the \emph{optimal solution set}, the set of points in $\mX(\btheta)$ that are optimal under $\btheta$.

\subsection{The inverse problem} \label{subsec:problem_def_inverse}

Consider a data set $\{\bhx_i\}_{i=1}^N$ of $N \geq 1$ decisions. These decisions may arrive from different versions of the forward model $\FOP_i(\btheta) := \min \{ f_i(\bx, \btheta) \;|\; \bx \in \Xfeas_i(\btheta) \}$, for example via an agent who solves the same decision-generating process several times under different conditions. 
We also index the corresponding optimal sets $\Xopt_i(\btheta)$. 
Given this data set, inverse optimization estimates a parameter vector $\btheta^*$ such that the aggregate ``fit'' of the  corresponding forward models $\FOP_i(\btheta^*)$ to $\bhx_i$ is optimized. An estimate $\btheta^*$ is considered a ``perfect fit'' for $\bhx_i$ if $\bhx_i \in \Xopt_i(\btheta^*)$. We formally define the set of perfect fit estimates as the \emph{inverse-feasible set}
\begin{align*}
    \Thetainv_i(\bhx_i) := \left\{ \btheta \;\Big|\; \bhx_i \in \Xopt_i(\btheta) \right\}.
\end{align*}
%

A fundamental question in the design of an inverse optimization problem is whether inverse-feasibility must be strictly enforced. This leads to what we call the \emph{classical} versus \emph{data-driven} formulations (see Figure \ref{fig:summary_of_io_methods}).

\begin{figure}[t]
    \centering
        \scalebox{0.9}{%
        \begin{tikzpicture}[font=\sffamily]
        
        \node(FO) [process, text width=4.1cm] {Forward model and $\bTheta$};
        
        \node(Ass) [process, below of = FO, yshift=-0.5cm, text width=6.8cm] {Must inverse-feasibility be satisfied?}; 
        

        \node(CIO) [process, below of = Ass, xshift=-4.5cm, yshift=-1.1cm, text width=7cm] {Classical Inverse Optimization \\ (Section~\ref{sec:classic_IO})};
        \node(DDIO) [process, below of = Ass, xshift=4.5cm, yshift=-1.1cm, text width=7cm] {Data-driven Inverse Optimization \\ (Section~\ref{sec:data-driven_IO})};

        \node(prior) [noBox, below of = CIO, yshift=-0.35cm, text width=5.3cm, minimum height=0.5cm] {Add constraints $\bhx_i \in \Xopt_i(\btheta)$};
        \node(formcio) [noBox, below of = prior, yshift=-0.35cm, text width=7.1cm, minimum height=0.5cm] {Solve inverse problem by reformulating $\Xopt_i(\btheta)$ using the forward model structure};

        \node(loss) [noBox, below of = DDIO, yshift=-0.35cm, text width=5.3cm, minimum height=0.5cm, align=justify] {Add loss functions $\ell(\bhx_i, \Xopt_i(\btheta))$};
        \node(formddio) [noBox, below of = loss, yshift=-0.35cm, text width=7.1cm, minimum height=0.5cm] {Solve inverse problem by additionally considering the empirical risk of the loss};

        \draw[arrow, double] (FO) to (Ass);
        \draw[arrow, double, out=180, in=90] (Ass) to (CIO) node[above left, yshift=0.7cm] {Yes};
        \draw[arrow, double, out=0, in=90] (Ass) to (DDIO) node[above right, yshift=0.7cm] {No};
        
        \draw[arrow] (CIO) to (prior);
        \draw[arrow] (prior) to (formcio);
        
        \draw[arrow] (DDIO) to (loss);
        \draw[arrow] (loss) to (formddio);

        \node[dashedBorder, fit=(CIO)(prior)(formcio)] {};
        \node[dashedBorder, fit=(DDIO)(loss)(formddio)] {};
        
        \end{tikzpicture}
        }
    \caption{Classical and data-driven inverse optimization methods are differentiated by whether inverse feasibility must be satisfied. From there, classical and data-driven techniques follow their respective steps.} 
    \label{fig:summary_of_io_methods}
\end{figure}
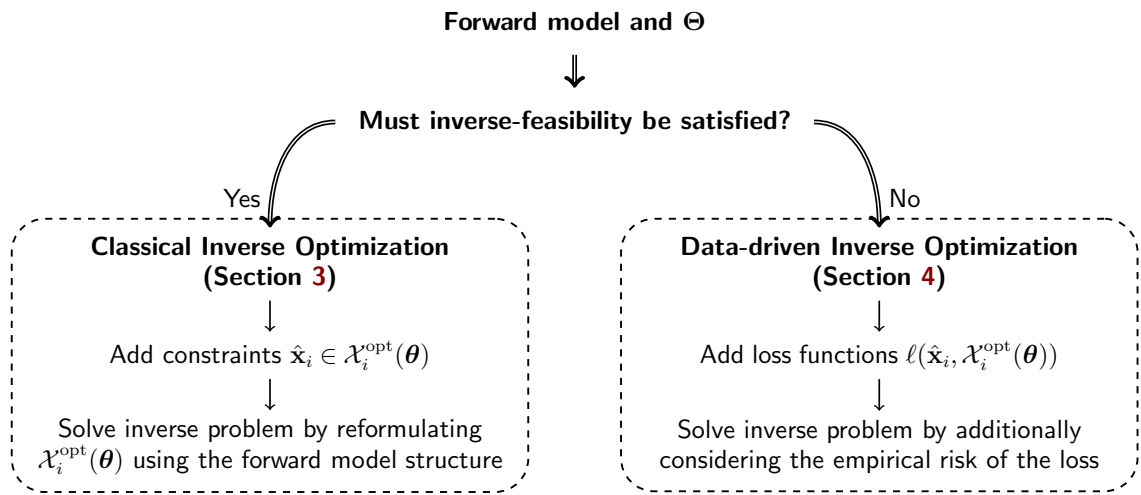

\paragraph{Classical inverse optimization (Section~\ref{sec:classic_IO}).} 
These inverse optimization problems solve for a parameter $\btheta \in \bTheta$ that is inverse-feasible and minimizes an application-specific objective $h(\btheta)$:
\begin{align}\label{eq:classical_io_highlevel}
    \min_{\btheta} \;\left\{  h(\btheta) \;\Big|\; \btheta \in \Thetainv_i(\bhx_i) \; \forall i \in \{1, \ldots, N\},\;  \btheta \in \bTheta \right\}.
\end{align}
The inverse-feasibility constraints make~\eqref{eq:classical_io_highlevel} a bilevel program. Many techniques have been developed to reformulate it into a single-level form.

\paragraph{Data-driven inverse optimization (Section~\ref{sec:data-driven_IO}).} 
Instead of enforcing inverse-feasibility, loss functions $\ell(\bx_i, \Xopt_i(\btheta))$ are used to penalize the extent to which inverse-feasibility is violated: 
%
\begin{align}\label{eq:data-driven_io_highlevel}
    \min_{\btheta} \;\bigg\{ \kappa h(\btheta) + \frac{1}{N} \sum_{i=1}^N \ell\left(\bhx_i, \Xopt_i(\btheta) \right) \; \Big| \;  \btheta \in \bTheta \bigg\}.
\end{align}
Here, $\kappa \geq 0$ describes a trade-off between an application-specific objective $h(\btheta)$ and the loss. 

\paragraph{Alternative learning paradigms (Section~\ref{sec:related_model_paradigms}).} 
The above problems become computationally challenging with large forward models or data sets, necessitating other approaches that permit fast algorithms. We discuss Inverse Online Learning (IOL), which considers input data sequentially, and Inverse Reinforcement Learning (IRL), which addresses large sequential decision-making models.

\subsection{Choosing an inverse model}

The choice of which inverse optimization model to use depends on the application. Most applications in the literature can be described as either \emph{design} or \emph{estimation} problems. 

\paragraph{Design applications (Section \ref{sec:classic_IO_apps}).} 
Inverse optimization can be used to design prices, incentives, or mechanisms $\btheta \in \bTheta$ to induce particular agent decisions $\{\bhx_i\}_{i = 1}^N$ by making them optimal under the computed parameters. These problems assume that the agents' decision-generating processes are known and focus on classical inverse optimization techniques. 

\paragraph{Estimation applications (Section \ref{sec:modern_applications}).} These problems learn a decision-making model using decision inputs representing observed agent behavior. 
Here, data-driven inverse optimization models can generate parameter estimates $\btheta \in \bTheta$ even if it is infeasible to render all observed decisions $\{\bhx_i\}_{i = 1}^N$ optimal. This property is important in applications with large decision data sets obtained from noisy measurements, inconsistent behavior, or unobserved contextual factors. 

\section{Classical Inverse Optimization}\label{sec:classic_IO}

In this section, we overview reformulation techniques for the inverse-feasibility constraint in 
Problem \eqref{eq:classical_io_highlevel}. 
We present different techniques for linear, conic, discrete, and sequential forward optimization models that each reformulate the corresponding inverse model into a tractable optimization problem.
We also present variations where the input is only given in a partial form.

Although the reformulation techniques presented in this section can be used to solve any instance of Problem~\eqref{eq:classical_io_highlevel}, we consider the most common objective function found in the literature, i.e., $h(\btheta) = \| \btheta - \hat\btheta \|_p$, where $\bhtheta$ is a fixed value, interpreted as an a priori belief or estimate:
%
\begin{align*}
    \IOPp(\bhx, \bhtheta) := 
    \min_{\btheta} \;\left\{ \norm{\btheta - \bhtheta}_p \;\bigg|\; \btheta \in \Thetainv(\bhx), \; \btheta \in \bTheta \right\}.
\end{align*}
For ease of exposition, we also assume a single decision input ($N=1$). 

The earliest paradigms in classical inverse optimization used the following characterization of the inverse-feasible set (see~\citet{heuberger2004inverse} for a review):
\begin{align}\label{eq:xinv_fundamental}
    \Thetainv(\bhx) = \left\{ \btheta \;|\; f(\bhx, \btheta) \leq f(\bx, \btheta) , \; \forall \bx \in \Xfeas(\btheta) \right\}.
\end{align}
For $\bhx$ to be optimal, it must have an objective value that is at least as good as any other feasible solution. However,~\eqref{eq:xinv_fundamental} requires enumerating every feasible solution. Consequently, early literature focuses on integer and network flow forward models where this enumeration can be finite. For example,~\citet{burton1992instance, burton1994use} use~\eqref{eq:xinv_fundamental} to estimate arc costs in shortest path forward models,~\citet{zhang1996inverse} explore spanning tree models, and~\citet{guler2010capacity} estimate the capacity constraints of minimum cost flow models. 
Since enumerating all feasible solutions is impractical in general, alternative and tractable characterizations of $\Thetainv(\bhx)$ were developed. 

\subsection{Linear models}\label{subsec:linear}

Linear models are among the most well-studied in inverse optimization, with efficient techniques to estimate both objectives and constraints. Consider the Linear Forward Model
\begin{align*}
    \FOPL((\btheta, \bPhi, \bpsi)) := \min \{ \btheta^\tpose \bx \;|\; \bPhi \bx \geq \bpsi \}
\end{align*}
where $\btheta \in \field{R}^n$, $\bPhi \in \field{R}^{m \times n }$, and $\bpsi \in \field{R}^m$ are the objective vector, constraint matrix, and right-hand-side constraint vector, respectively. The inverse optimization literature for linear forward models generally considers the estimation of either the objective or the constraint parameters independently, i.e., under the assumption that the other components are given:
\begin{itemize}
    \item \textbf{Estimating the objective:} Let $\bPhi = \bA$ and $\bpsi = \bb$ where $\bA \in \field{R}^{m \times n}$ and $\bb \in \field{R}^m$ are known constraint parameters, and estimate an objective vector $\btheta$ of a minimum distance from $\bhtheta$.
    
    \item \textbf{Estimating the constraints:} Let $\btheta = \bc$ where $\bc \in \field{R}^n$ is a known objective vector and estimate $(\bPhi, \bpsi)$ of minimum distance from $(\hat\bPhi, \hat\bpsi)$. Note that this generalizes special cases of estimating only the constraint matrix (i.e., $\bTheta := \{ (\bPhi, \bpsi) \;|\; \bpsi = \bb \}$) or estimating the only the right-hand-side vector (i.e, $\bTheta := \{ (\bPhi, \bpsi) \;|\; \bPhi = \bA \}$).

\end{itemize}
%
When estimating only objective parameters, we assume $\bhx$ is feasible for the given constraint parameters, i.e., that $\bA \bhx \geq \bb$. When it is obvious from context, we simplify the notation of $\FOPL((\btheta, \bPhi, \bpsi))$ to omit known parameters (e.g., let $\FOPL(\btheta) := \min \{ \btheta^\tpose \bx \;|\; \bA \bx \geq \bb \}$).

\subsubsection{Estimating the objective.}

\citet{ahuja2001inverse} propose one of the earliest inverse optimization frameworks for estimating the objective of general bounded linear forward models $\FOPL(\bhtheta)$. They show that the inverse optimization problem of a linear program can itself be a linear program by characterizing $\Thetainv(\bhx)$ using the complementary slackness conditions in linear programming. Equivalently, strong duality conditions can also be used \citep{chan2018inverse, shahmoradi2021quantile, ghobadi2021inferring}. 
\begin{property}[Complementary Slackness]\label{prop:comp_slack}
    If $\bhx$ is an optimal solution for $\FOPL(\btheta)$, then there exists a dual vector $\blambda$ that satisfies (i) dual feasibility $\bA^\tpose \blambda = \btheta, \; \blambda \geq \bzero$ and (ii) complementary slackness $(\bA \bhx - \bb)^\tpose \blambda = 0$. 
\end{property}
From Property \ref{prop:comp_slack}, the inverse-feasible set of $\FOPL(\btheta)$ has a convenient linear representation $\Thetainv(\bhx) = \{ \btheta \;|\; \exists \blambda \geq \bzero : \bA^\tpose \blambda = \btheta \;,\;  (\bA \bhx - \bb)^\tpose \blambda = 0 \}$. The inverse optimization problem then becomes
\begin{align}
\label{eq:ahuja_ilo}
\begin{split}
    \min_{\btheta, \blambda} \quad & \norm{\btheta - \bhtheta}_p \\
    \st \quad   & \bA^\tpose \blambda = \btheta \\
                & \left(\bA \bhx - \bb \right)^\tpose \blambda = 0 \\
                & \btheta \in \bTheta, \quad  \blambda \geq \bzero.
\end{split}
\end{align}
If $p=1$ or $p=\infty$, then problem~\eqref{eq:ahuja_ilo} becomes a linear program.

We next highlight the strong duality approach, which leads to an equivalent formulation.
\begin{property}[Strong Duality]\label{prop:strong_duality}
    If $\bhx$ is an optimal solution for $\FOPL(\btheta)$, then there exists a dual vector $\blambda$ that satisfies (i) dual feasibility $\bA^\tpose \blambda = \btheta, \; \blambda \geq \bzero$ and (ii) strong duality $\btheta^\tpose \bhx = \blambda^\tpose \bb$. 
\end{property}
Property \ref{prop:strong_duality} implies that we can replace the constraint $\left(\bA \bhx - \bb \right)^\tpose \blambda = 0$ in Problem \eqref{eq:ahuja_ilo} with $\btheta^\tpose \bhx = \blambda^\tpose \bb$ via the following transformation
%
\begin{align*}
    \left(\bA \bhx - \bb \right)^\tpose \blambda = \blambda^\tpose \bA \bhx - \bb^\tpose \blambda = \btheta^\tpose \bhx - \bb^\tpose \blambda.
\end{align*}
This observation underscores that there are often multiple equivalent formulations of an inverse problem for the same forward model, based on the choice of optimality condition that is used to develop the inverse problem. For many classical inverse problems, where inverse-feasibility is a constraint, these alternatives may often be considered equivalent. On the other hand, the specific choice of optimality condition can motivate different loss functions and lead to different model formulations in data-driven inverse optimization, as will be discussed in Section \ref{sec:data-driven_IO}.


As an extension, it has been shown that the inverse problem of an infinite-dimensional linear forward model is an infinite-dimensional convex program. \citet{ghate2015inverse} study linear forward problems with a countably infinite number of variables and constraints, 
where under a sufficiency condition for strong duality, they construct $\Thetainv(\bhx)$ as an infinite number of dual feasibility constraints and propose a convergent solution algorithm for the resulting infinite-dimensional inverse problem.
\citet{nourollahi2019inverse} extend these results to minimum cost flow problems on infinite dimensional networks and \citet{ghate2020inverse} considers semi-infinite linear forward models that feature a finite dimensional decision vector but an uncountably infinite constraint set.

\subsubsection{Estimating constraint parameters.}
\label{sec:classical_io_linear_estimating_constraints}

\citet{chan2020inverse} consider $\FOPL(\bPhi)$ where the right-hand-side constraint vector $\bpsi = \bb$ is known and only $\bPhi$ must be estimated. Given an observed $\bhx$, the inverse-feasible set is 
\begin{align*}
    \Thetainv(\bhx) = \{ \bPhi \;|\; \exists \blambda \geq \bzero : \bPhi^\tpose \blambda = \bc , \; \bc^\tpose \bhx = \blambda^\tpose \bb \}.
\end{align*}
Although $\Thetainv(\bhx)$ contains bilinear constraints $\bPhi^\tpose \blambda = \bc$,~\citet{chan2020inverse} propose a solution method by observing that for $\bhx$ to be optimal, it must satisfy at least one constraint with equality. Specifically, an optimal solution of the inverse problem can be obtained by perturbing the nearest facet $\{ \bx \;|\; \hat{\bphi}_j^\tpose \bx \geq b_j \}$ from $\bhx$ until the corresponding constraint is satisfied. 

\citet{ghobadi2021inferring} consider the general problem of estimating both $\bPhi$ and $\bpsi$. 
Here, the inverse-feasible set $\Thetainv(\bhx) = \{ (\bPhi, \bpsi) \;|\; \exists \blambda \geq \bzero : \bPhi^\tpose \blambda = \bc , \; \bc^\tpose \bhx = \blambda^\tpose \bpsi \}$
%
%
is now bilinear in two constraints. 
They derive a simple convex formulation of the inverse optimization problem by noting that the optimality of $\bhx$ means that $\bc^\tpose \bx \geq \bc^\tpose \bhx$ is an implicit constraint. 
Since for a linear optimization problem, all optimal solutions must lie on at least one of the facets of the feasible set, they select the implicit constraint as the supporting facet and reduce the inverse problem to estimating parameters $(\bPhi, \bpsi)$ such that the observed decision $\bhx$ is feasible:
\begin{align*}
    \min_{\bPhi, \bpsi} \quad   & \norm{ \left(\bPhi, \bpsi) - (\hat\bPhi, \hat\bpsi\right) }_p \\
    \st \quad                   & \bPhi^\tpose \bhx \geq \bpsi \\
                                & \left(\bPhi, \bpsi\right) \in \bTheta.
\end{align*}
%

The above approaches are also useful in robust optimization. For instance,~\citet{chan2020inverse} estimate parameters in an uncertainty set for robust linear optimization models of the form
\begin{align}\label{eq:fop_ro}
    \min_{\bx} \max_{\bA \in \set{U}(\btheta)} \left\{ \bc^\tpose \bx \;\Big|\; \bA \bx \geq \bb \right\},
\end{align}
where the uncertainty set depends on the parameter to be estimated.
Since robust linear optimization problems with polyhedral uncertainty sets 
can be reformulated into single-level linear optimization problems of the structure in $\FOPL(\bPhi)$ with the uncertainty set parameters in the constraints of the linear program, the above approaches apply. 

\subsubsection{Jointly estimating objective and constraint parameters.} In some special cases, it is possible to
jointly estimate parameters from both the objective and constraints. Consider the forward model $\FOPL((\btheta, \bpsi))$ where the constraint matrix $\bPhi = \bA$ is given, but the objective vector and right-hand-side vector need to be estimated. To solve this problem, we can leverage the complementary slackness formulation~\eqref{eq:ahuja_ilo}. Note here that if $\bb$ in problem~\eqref{eq:ahuja_ilo} is a decision variable (now denoted as $\bpsi$), the complementary slackness equation becomes bilinear, which can be reformulated into linear constraints using an additional set of auxiliary binary variables $\bz$. This implies that the inverse optimization problem can be solved as the mixed integer linear program 
\begin{align}\label{model:LP_joint_estimation}
\begin{split}
\min_{\btheta, \blambda, \bpsi, \bz} \quad      & \norm{ \left(\btheta, \bpsi\right) - \left(\hat\btheta, \hat\bpsi\right)  }_p \\
\st \quad\;  & \bA \bhx - \bpsi \geq \bzero \\
& \bA\bhx - \bpsi \leq M\bz  \\
& \blambda \leq M(\mathbf{1}-\bz) \\
& \bA^\top \blambda = \btheta  \\
& (\btheta, \bpsi) \in \bTheta, \quad \blambda \geq \bzero, \quad \bz \in \{0,1\}^m
\end{split}
\end{align}
where $M > 0$ is a sufficiently large constant.

\subsection{Conic and convex models} 
\label{sec:classical_IO_convex}

For $m \geq 2$, let $\set{C} \subset \field{R}^m$ be a proper cone, meaning it is convex and closed, it has a non-empty interior, and $\{\bx, -\bx\} \in \set{C} \Rightarrow \bx = \bzero$. 
Consider the Conic Forward Model
\begin{align}\label{eq:fop_conic}
    \FOPC(\btheta) := \min_{\bx} \; \left\{ f(\bx, \btheta) \;|\; -\bg(\bx, \btheta) \in \set{C} \right\}
\end{align}
where $f(\bx, \btheta)$ is a convex function in $\bx$ and $\bg(\bx, \btheta) = \big( g_1(\bx, \btheta), \dots, g_m(\bx, \btheta) \big)$ is a vector-valued function whose elements are each differentiable and convex. 
Note that setting $\set{C} = \field{R}^m_+$ reduces the forward model to a classical convex optimization problem, where the constraint $-\bg(\bx, \btheta) \in \mC$ in model \eqref{eq:fop_conic} can be equivalently written as $\bg(\bx, \btheta) \leq \mathbf{0}$. \citet{iyengar2005inverse} derive a tractable reformulation of $\IOPp(\bhx, \bhtheta)$ using the Karush-Kuhn-Tucker (KKT) conditions.
\begin{property}[KKT Conditions]
    A decision $\bhx$ is optimal for $\FOPC(\btheta)$ if (i) there exists $\bx \in \set{C}$ for which $-\bg(\bx, \btheta) \in \interior(\set{C})$, and (ii) there exists $\blambda \in \set{C}$ for which
    \begin{align*}
        \nabla_{\bx} f(\bhx, \btheta) + \sum_{j=1}^m \lambda_j \nabla_{\bx} g_j(\bhx, \btheta) = 0 & \textrm{ and } 
         \lambda_j g_j(\bhx, \btheta) = 0, \quad \forall j \in \{1, \dots, m\}.
    \end{align*}
\end{property}

Let the inverse-feasible set $\Thetainv(\bhx)$ be the set of $\btheta$ that satisfy the KKT conditions above.
For a general $\FOPC(\btheta)$, this set features bilinear equations in $\lambda_j$ and $g_j(\bhx, \btheta)$, meaning that the general inverse problem is not easily solvable. However, in the case where $\btheta$ lies only in the objective, the corresponding inverse problem is a conic program.
\begin{lemma}[\citet{iyengar2005inverse}]\label{lem:iyengar_inverse_conic}
    Consider $\FOPC(\btheta)$. Suppose that $\bg(\bx, \btheta) = \bg(\bx)$ and that there exists $\bx$ for which $-\bg(\bx) \in \interior(\set{C})$. Then $\IOPp(\bhx, \bhtheta)$ is equivalent to
    \begin{align*}
        \min_{\btheta, \blambda} \quad  & \norm{\btheta - \hat\btheta}_p \\
        \st \quad                       & \nabla_{\bx} f(\bhx, \btheta) + \sum_{j=1}^m \lambda_j \nabla_{\bx} g_j(\bhx) = \bzero \\
                                        & \lambda_j g_j(\bhx) = 0, \quad \forall j \in \{1, \dots, m\} \\
                                        & \btheta \in \bTheta, \quad \blambda \in \set{C}.
    \end{align*}
    Furthermore if $\nabla_{\bx} f(\bhx, \btheta)$ is an affine function of $\btheta$, then the above problem is convex.
\end{lemma}
%

Specific characterizations of $\FOPC(\btheta)$ can incur very efficient algorithms.
For example,~\citet{zhang2010augmented} consider quadratic forward optimization models where $f(\bx, \btheta) = \bx^\tpose \bPhi \bx + \bpsi^\tpose \bx$ for $\bTheta =  \{ (\bPhi, \bpsi)  \;|\; \bPhi \in \field{S}^{n \times n}, \; \bpsi \in \set{R}^n \}$. 
Using Lemma~\ref{lem:iyengar_inverse_conic}, the inverse conic problem is a semi-definite program (SDP). 
Here,~\citet{zhang2010augmented} propose a more efficient algorithm by first writing the dual of this program as a semi-smooth, differentiable convex program that can be solved by a Newton method. Further,~\citet{zhang2010inverse} consider convex forward models with separable basis objective functions $f(\bx,\btheta) = \sum_{b=1}^B \theta_b f^{(b)}(\bx)$ and known linear equality constraints. In this case, the inverse problem is a linear program.

\subsection{Markov decision processes}
\label{subsec:MDP}

\citet{erkin2010eliciting} use inverse optimization to estimate the reward function in a finite-state finite-action Markov decision process (MDP). 
Consider an MDP defined by the tuple $(\set{S}, \set{A}, p, \btheta, \gamma)$, where $s \in \set{S}$ and $a \in \set{A}$ are states and actions in their respective sets, $p(s'|a, s) \in [0, 1]$ is the state transition probability for any state-action pair, $\btheta : \set{S} \times \set{A} \to \field{R}$ is a reward function, and $\gamma \in [0, 1]$ is a discount factor. In a finite-state finite-action MDP, we can characterize the reward function as a matrix $\btheta \in \field{R}^{|\set{S}| \times |\set{A}|}$ for which $\theta_{s, a}$ denotes the reward for a state-action pair. A (forward) MDP determines an optimal value function $\bv^*$ using the Bellman equations
\begin{align}\label{eq:bellman}
    v_s = \max_{a \in \set{A}} \left\{ \theta_{s, a} + \gamma \sum_{s' \in \set{S}} p(s'|a, s) v_{s'}  \right\}, \quad \forall s \in \set{S}.
\end{align}
To formulate the inverse problem, we first write the MDP as a linear program:
\begin{align*}
    \FOPMDP(\btheta) := \min_{\bv} \left\{ \sum_{s \in \set{S}} v_s \;\Bigg|\; v_s \geq \theta_{s, a} + \gamma \sum_{s' \in \set{S}} p(s'|a, s) v_{s'} ,\; \forall s \in \set{S}, \; a \in \set{A}   \right\}. 
\end{align*}

In an inverse MDP, rather than observing $\hat{\bv}$, we observe a policy $\hat\bpi : \set{S} \rightarrow \set{A}$ where $\hat\bpi(s) \in \argmax_{a \in \set{A}} \{ \theta_{s, a} + \gamma \sum_{s' \in \set{S}} p(s'|a, s) v_{s'}  \}$.
Since an optimal policy can be derived by solving the dual of the above linear program~\citep{puterman1990markov}, it also satisfies complementary slackness.
\begin{property}[Complementary Slackness for MDPs]\label{prop:cs_mdp}
    The policy $\bhpi$ is optimal for $\FOPMDP(\btheta)$ if and only if 
    \begin{align*}
        \bhpi(s) = a' \; \Longrightarrow \; v_s = \theta_{s, a'} + \gamma \sum_{s' \in \set{S}} p(s'|a', s) v_{s'}, \quad \forall s \in \set{S}.
    \end{align*}
\end{property}

The inverse-feasible set of $\FOPMDP(\btheta)$ is
\begin{align*}
    \Thetainv(\bhpi) = \left\{ \btheta \;\vast|\; \exists \bv:
    \begin{array}{ll} \displaystyle
        v_s \geq \theta_{s, a} + \sum_{s' \in \set{S}} p(s'|a, s) v_{s'} , \quad \forall s \in \set{S}, \; a \in \set{A} \\ \displaystyle
        v_s = \theta_{s, a'} + \sum_{s' \in \set{S}} p(s'|a', s) v_{s'} , \quad \forall s \in \set{S}, \; a'= \bhpi(s)
    \end{array}
    \right\}
\end{align*}
and $\IOPp(\bhx, \bhtheta)$ is the following convex program 
\begin{subequations} \label{eq:inverse_mdp}
\begin{align}
    \min_{\btheta, \bv} \quad & \norm{\btheta - \hat\btheta}_p \\
    \st \quad   & v_s \geq \theta_{s, a} + \gamma \sum_{s' \in \set{S}} p(s'|a, s) v_{s'},  \quad \forall s \in \set{S}, \; a \in \set{A} \label{eq:inverse_mdp2} \\ 
                & v_s = \theta_{s, a'} + \gamma \sum_{s' \in \set{S}} p(s'|a', s) v_{s'},  \quad \forall s \in \set{S}, \; a'= \bhpi(s)  \label{eq:inverse_mdp3} \\
                & \btheta \in \bTheta.  \label{eq:inverse_mdp4}
\end{align}
\end{subequations}

If the state or action sets are infinite, then $\FOPMDP(\btheta)$ is an infinite-dimensional linear program and we can use infinite-dimensional inverse linear optimization~\citep{ghate2015inverse, ghate2020inverse}. Large scale MDPs are often solved with approximate dynamic programming (ADP) or reinforcement learning. 
Inverse methods for ADP forward models remain an unexplored research area. 
The inverse reinforcement learning literature has developed in parallel with inverse optimization with limited cross-pollination~\citep{abbeel2004apprenticeship}. 

\subsection{Discrete models}\label{subsec:integer}

We next consider forward models that are mixed integer linear programs (MILPs), which do not possess optimality certificates that can be conveniently represented in a small number of equations (e.g., the KKT conditions). 
Consider the Mixed Integer Forward Model 
\begin{align*}
    \FOPMIP(\btheta) := \max_{\bx} \left\{ \btheta^\tpose \bx \;\Big|\; \bA \bx \leq \bb, \; \bx \in \field{R}^{n-q} \times \field{Z}^{q} \right\}.
\end{align*}
%
Note that we use a maximization problem to remain consistent with the original literature. Inverse optimization methods for $\FOPMIP(\btheta)$ either use certificates of strong duality for integer programming, analogous to inverse linear optimization techniques---here, it leads to inverse problems with an exponential number of variables and constraints---or use cutting plane algorithms.  

\subsubsection*{Duality via sub-additive functions.}

\citet{schaefer2009inverse} generalizes the previous strong duality approach to inverse integer linear optimization where $q=n$ in $\FOPMIP(\btheta)$. 
\citet{lamperski2015polyhedral} further explore MILPs where $q < n$. 
We highlight the strictly integer case below.

While the dual vector of a linear program is a non-negative vector, the dual ``variable'' of an integer program $\FOPMIP(\btheta)$ is a non-decreasing and super-additive function $F:\field{Z}^m \rightarrow \field{R}$~\citep{lasserre2009linear}.
An optimal dual function must be dual feasible and achieve the primal optimal value.
\begin{property}[Strong Duality for MIPs]
    Let $\bA_i\in \field{R}^n$ be the $i$-th column of $\bA$.  
    Then, $\bhx$ is optimal for $\FOPMIP(\btheta)$ if there exists a non-decreasing super-additive function $F: \field{R}^m \rightarrow \field{R}$ where 
    %
    \begin{align} \label{eq:duality_mips}
        F(\bzero) = 0, \quad F(\bb) = \btheta^\tpose \bhx, \quad F(\bA_i) \leq \theta_i \qquad \forall i \in \{1, \dots, n\}.
    \end{align}
\end{property}

To characterize the inverse-feasible set, we represent the dual function as a vector whose elements represent the function output for points in $\field{Z}^m$, similar to describing policy functions in MDPs. 
Although $\field{Z}^m$ is infinite,~\citet{schaefer2009inverse} shows that we only need to consider a sufficiently large finite subset $\set{B} \subset \field{Z}^m$ as the domain of the dual function. Given $\set{B}$, the inverse-feasible set is
\begin{align*}
    \Thetainv(\bhx) = \left\{ \btheta \;\vast|\; \exists F \in \field{R}^{|\set{B}|} \geq \bzero \, : \, 
    \begin{array}{l} \displaystyle
        \text{Constraints \eqref{eq:duality_mips}} \\ 
        F(\bbeta) \leq F(\bbeta'), \quad \forall \bbeta \leq \bbeta' \in \set{B} \\
        F(\bbeta) + F(\bbeta') \leq F(\bbeta + \bbeta'),  \quad \forall \bbeta, \bbeta' \in \set{B}
    \end{array}
    \right\}
\end{align*}
where the last two sets of inequalities enforce super-additivity.
\citet{schaefer2009inverse} then develops sufficient conditions on $\set{B}$ such that this $\Thetainv(\bhx)$ fully describes inverse-feasibility. 
\begin{theorem}[\citet{schaefer2009inverse}]\label{thm:inverse_integer_schaefer}
    Suppose that (i) $\bA \geq \bzero$ and $\bb \geq \bzero$ componentwise, and (ii) $\bb \geq \bA_i$ for all $i \in \{1, \dots, n\}$. Let $\bar{b} = \max_i \{ b_i \}$ and define the lattice $\set{B} := \field{Z}^n \times ( \otimes_{i=1}^n [0, \bar{b}] )$. Then, $\IOPp(\bhx, \bhtheta)$ is equivalent to the following (exponentially sized) convex optimization problem:
    \begin{align*}
        \min_{\btheta, F} \quad & \norm{\btheta - \hat\btheta}_p \\
        \st \quad   & F(\bA_i) \geq \theta_i,  \quad \forall i \in \{1, \dots, n \} \\
                    & F(\bb) = \btheta^\tpose \bhx , \quad F(\bzero) = 0 \\
                    & F(\bbeta) \leq F(\bbeta'), \quad \forall \bbeta \leq \bbeta' \in \set{B} \\
                    & F(\bbeta) + F(\bbeta') \leq F(\bbeta + \bbeta') , \quad  \forall \bbeta, \bbeta' \in \set{B} \\
                    & \btheta \in \bTheta, \quad F \geq \bzero
    \end{align*}
\end{theorem}

The inverse problem to a mixed integer forward optimization model is a convex program whose non-linearities depend only on $\bTheta$. 
However, this program requires an exponential number of variables $F \in \field{R}^{|\set{B}|}$ and an exponential number of constraints. If the dimension $n$ of the forward problem is large, then this super-additive duality-based approach is likely to be intractable. One approach is to approximate the class of super-additive functions $F(\bbeta)$. For example,~\citet{turner2013examining} use linear and quadratic approximations, often obtaining an optimal solution for small problems. 

\subsubsection*{Cutting plane algorithms.}

An alternative approach is to directly leverage the fundamental definition of $\Thetainv(\bhx)$ in~\eqref{eq:xinv_fundamental}, and write 
$\IOPp(\bhx, \bhtheta)$ as
\begin{align}\label{eq:IOprior_fullMIP}
    \min_{\btheta} \;\left\{ \norm{\btheta - \hat\btheta}_p  \;\Big|\; \btheta^\top \bhx \geq \btheta^\top \bx \quad \forall \bx \in \extreme\left(\{ \bx \in \field{R}^{n-q} \times \field{Z}^q \;|\; \bA \bx \leq \bb  \}\right) ,\;  \btheta \in \bTheta \right\}.
\end{align}
Whereas~\eqref{eq:xinv_fundamental} was defined using the entire feasible set, i.e., $\btheta^\tpose \bhx \geq \btheta^\tpose \bx$ for all $\bx \in \Xfeas$, with a mixed integer linear forward problem, we only need to consider this optimality condition over the set of extreme points of the forward feasible set \citep{wang2009cutting}. Since this set is finite, the feasible set of Problem~\eqref{eq:IOprior_fullMIP} contains a finite but (generally) exponential number of constraints.

\citet{wang2009cutting} proposes a cutting plane algorithm that iteratively generates new extreme points and computes $\btheta$ to make $\bhx$ optimal over the growing set of previously generated points. 
Since the feasible set of $\FOPMIP(\btheta)$ may contain an exponential number of extreme points, each of being computationally demanding to compute, subsequent research focused on accelerating the cutting plane algorithm. \citet{duan2011heuristic} describe a parallel computing approach where different extreme points are simultaneously computed and added. \citet{bodur2022inverse} demonstrate that in many problem classes, generating extreme points or even interior points near $\bhx$ may be sufficient. This observation motivates the development of a trust-region based cutting plane algorithm where smaller, restricted forward problems are solved to generate both better and faster cuts. 

\subsection{Partially constrained inverse problems}\label{subsec:sec3_miscallenous}

Some studies assume that only a lower-dimensional function of a decision vector is given, such as a partial solution or the objective value of a solution. 

\subsubsection*{Inverse optimization with a partial solution.}\label{subsubsec:Partial_IO} 
Several papers have examined the problem of estimating model parameters when only a ``part" of a decision is observed. Examples appear in various network flow problems where we observe flows along only a subset of the arcs. For example, \citet{yang2007partial} consider inverse assignment problems where only a subset of the assignments are identified, while~\citet{cai2008partial} explore  minimum spanning trees where only a subset of disconnected arcs is given. The input to these inverse problems is no longer a point $\bhx \in \mX$, but rather a set $\hat \mX \subseteq \mX$ where a subset of variables denoted by $\mM$ are fixed to known values:
\begin{align*}
    \hat \mX = \left \{\bx  \; | \; \bx \in \mX, \; x_i = \hat{x}_i \; \forall i \in \mM \right \}.
\end{align*}
The Partial Inverse Optimization problem seeks a parameter $\btheta$ such that at least one $\bx \in \hat\mX$ is an optimal solution \citep[e.g.][]{zhang2016algorithms, li2018partial, tayyebi2020partial}:
\begin{align}\label{eq:partial_IO}
 \min_{\btheta, \bx} \;\left\{ \norm{\btheta - \hat\btheta}_p \;\Big|\; \bx \in \hat\mX \;,\; \btheta \in \Thetainv(\bx) \;,\; \btheta \in \bTheta \right\}.
\end{align}

For linear models $\FOPL(\btheta)$, problem~\eqref{eq:partial_IO} can be written as the following bilinear program using the complementary slackness property
\begin{align*}
\min_{\btheta, \blambda, \bx} \quad & \norm{\btheta - \hat\btheta}_p \\
\st \quad  
& (\bA \bx - \bb)^\top \blambda = \bzero \\
& \blambda^\top \bA = \btheta \\
& \bA \bx \geq \bb\\
& \btheta \in \bTheta, \quad \bx \in \hat\mX, \quad \blambda \geq \bzero.
\end{align*}
This model can be reformulated as a mixed integer linear program or solved using decomposition algorithms \citep{hu2012linear}. \citet{wang2013branch} develops a solution method to solve the Partial Inverse Optimization problem for mixed integer linear forward models $\FOPMIP(\btheta)$. The solution method builds upon the cutting plane approach presented in Section \ref{subsec:integer}.

\subsubsection*{Inverse optimal value.}\label{subsubsec:IOV}
\citet{ahmed2005inverse}, \citet{mostafaee2016inverse} and \citet{vcerny2016inverse} 
consider the Inverse Optimal Value problem. Rather than observing a decision $\bhx$, the input is the objective function value $\hat{z}$. This leads to the problem
\begin{align*}
    \min_{\btheta} \left\{ \norm{\btheta - \hat\btheta}_p \;\Big|\; \min_{\bx} \{ \btheta^\tpose \bx \;|\; \bA \bx \geq \bb \} =\hat{z}, \; \btheta \in \bTheta \right\}.
\end{align*}
Under a few mild assumptions on the forward model including a bounded feasible region, \citet{ahmed2005inverse} show that the inverse problem can be reformulated into 
\begin{align*}
    \min_{\btheta, \blambda} \quad & \norm{\btheta - \hat\btheta}_p  \\
    \st \quad   & \blambda^\tpose \bb = \hat{z} \\
                & \bA^\tpose \blambda = \btheta \\
                & \btheta \in \bTheta,  \quad \blambda \geq \bzero.
\end{align*}
The literature referenced in the previous paragraph also considers cases where there does not exist a $\btheta \in \bTheta$ such that $\btheta^\top \bhx = \hat{z}$, and methods are proposed to minimize this gap. 

 \begin{figure}[t]
    \centering
        \resizebox{0.95\linewidth}{!}{%
        \begin{tikzpicture}[font=\sffamily]
        
        \node(LPFO) [process, text width=4.1cm] {Linear Forward Models};
        
        \node(LPObj) [noBox, below of = LPFO, yshift=-1cm, xshift=-1.6cm, text width=3.0cm] {Estimating objectives \\ 
        \footnotesize\citep{ahuja2001inverse}};
        \node(LPCon) [noBox, below of = LPFO, yshift=-1cm, xshift=1.6cm, text width=4.0cm] {Estimating constraints \\ \footnotesize \citep{chan2020inverse, ghobadi2021inferring}};
        
        \node(CPFO) [process, left of = LPFO, xshift=-7cm, text width=4.1cm] {Conic Forward Models};
        \node(CPFOprop) [noBox, below of = CPFO, text width=4cm] {\footnotesize \citep{iyengar2005inverse, zhang2010augmented, zhang2010inverse}};

        \node(IPFO) [process, right of = LPFO, xshift=7cm, text width=3.6cm] {Integer Forward \\ Models};
        
        \node(IPDual) [noBox, below of = IPFO, yshift=-1cm, xshift=-1.8cm, text width=3.5cm] {Duality \\ \footnotesize \citep{schaefer2009inverse, lamperski2015polyhedral} };
        \node(IPCP) [noBox, below of = IPFO, yshift=-1cm, xshift=1.8cm, text width=3.5cm] {Cutting planes \\ \footnotesize  \citep{wang2009cutting, duan2011heuristic, bodur2022inverse} };

        \node(INFLPFO) [process, below of = LPObj, text width=4.6cm, yshift=-4cm] {Infinite-dimension \\ Linear Forward Models};
        \node(INFprop) [noBox, below of = INFLPFO, text width=4.2cm] {\footnotesize \citep{ghate2015inverse, nourollahi2019inverse, ghate2020inverse} };

        \node(NFFO) [process, above of = LPFO, xshift=5cm, yshift=4cm, text width=3.5cm] {Network Flow \\ Forward Models};
        \node(NFFOcite) [noBox, below of = NFFO, yshift=-0.8cm, text width=7.5cm] {\footnotesize \citep{burton1992instance, burton1994use, zhang1996calculating, zhang1996inverse, guler2010capacity} \\
        (Surveyed in~\citep{heuberger2004inverse})};

        \node(MDPFO) [process, below of = CPFO, yshift=-3.25cm, text width=3.8cm] {MDP Forward \\ Models};
        \node(MDPFOcite) [noBox, below of = MDPFO] {\footnotesize \citep{erkin2010eliciting}};

        \node[dashedBorder, fit=(LPFO)(LPObj)(LPCon)] {};
        \node[dashedBorder, fit=(IPFO)(IPDual)(IPCP)] {};
        
        \draw (LPFO) to (LPObj);
        \draw (LPFO) to (LPCon);
        \draw (IPFO) to (IPDual);
        \draw (IPFO) to (IPCP);
        
        \draw[conceptarrow, double] (LPFO) to [out=120, in=60] node[above] {\footnotesize KKT vs duality} (CPFO);
        \draw[arrow, double] (LPObj) to [out=-60, in=240] node[below] {\footnotesize via superadditive functions} (IPDual);
        \draw[arrow, double] (LPObj) to node[left, text width=2cm, text centered, align=center, xshift=0.05cm] {\footnotesize via semi-infinite duality} (INFLPFO);
        \draw[arrow, double] (NFFO) to [out=190, in=90] node[left, align=right, text centered, text width=2.8cm, yshift=0.5cm, xshift=0.75cm] {\footnotesize LP optimality certificates} (LPFO);
        \draw[conceptarrow, double] (NFFO) to [out=-10, in=60] node[right, align=left, text width=2.8cm, text centered] {\footnotesize \eqref{eq:xinv_fundamental} definition of optimality} (IPCP);
        
        \draw[conceptarrow, double] (LPObj) to [out=180, in=90] node[left, text width=2.5cm, yshift=-0.6cm, xshift=2cm, align=center, text centered] {\footnotesize complementary slackness} (MDPFO);
        
        \draw[arrow, double] ([yshift=-0.75cm]MDPFO.south) to [out=270, in=180] node[left, text width=2.5cm, yshift=-0.15cm, align=center, text centered] {\footnotesize infinite states/actions} (INFLPFO);

        
        
        \end{tikzpicture}
        }
    \caption{A roadmap of the main branches of classical inverse optimization. Bold arrows denote extensions and generalizations and dashed arrows denote conceptual similarities.}
    \label{fig:summary_of_section3}
\end{figure}
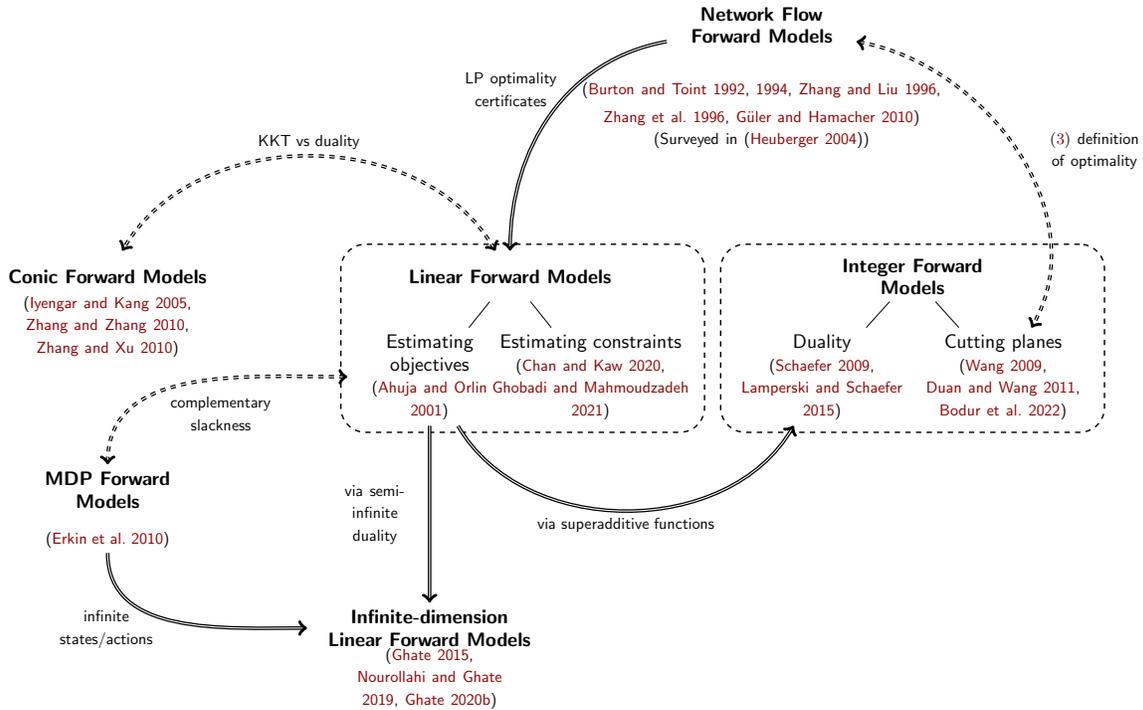

\subsection{Summary}

In this section, we examine classical inverse models for a variety of forward optimization structures. For each forward problem, we show that the inverse models can be formulated and solved by characterizing the optimality conditions using either forward optimality (i.e., via equation \eqref{eq:xinv_fundamental}) or a form of duality (i.e., complementary slackness, strong duality, KKT, or superadditive duality). Figure~\ref{fig:summary_of_section3} summarizes these results and highlights the relationship between the forward problem structure and the optimality conditions that are leveraged to solve the inverse problem.

\section{Data-Driven Inverse Optimization}\label{sec:data-driven_IO}


In this section, we survey different loss functions and solution algorithms in data-driven inverse optimization. Most of this literature explores convex forward models whose instances are associated with observed input parameters $\bu \in \set{U}$. For a given input $\bhu_i$, consider the convex forward model
\begin{align}\label{eq:fop_cvx_u}
    \FOPCVX_i(\btheta) := \min_{\bx}  \left\{ f(\bx, \bhu_i, \btheta) \;|\; \bg(\bx, \bhu_i, \btheta) \leq \bzero \right\}.
\end{align}
We can also define $f_i(\bx, \btheta) \equiv f(\bx, \bu_i, \btheta)$ and $\bg_i(\bx, \btheta) \equiv \bg(\bx, \bu_i, \btheta)$, meaning that introducing $\bu_i$ does not lose any generality. In practice, $\bu_i$ captures observable changes between different instances. 

In the data-driven problem, we observe a data set of decisions and optimal solution sets $\dataset := \{ (\bhx_i, \Xopt_i(\btheta)) \}_{i=1}^N$, drawn i.i.d.~from a probability distribution $(\bx, \Xopt(\btheta)) \sim \field{P}$.
Note that the following statements are equivalent: 
(i) we observe $\dataset := \{ (\bhx_i, \Xopt_i(\btheta)) \}_{i=1}^N$; 
(ii) we observe a data set of decisions and forward models $\dataset := \{ (\bhx_i, \FOPCVX_i(\btheta)) \}_{i=1}^N$;
(iii) we observe a data set of decisions and forward model inputs $\dataset := \{(\bhx_i, \bhu_i) \}_{i=1}^N$.
At times, it is advantageous to specifically assume the data is obtained a specific format, so we interchange these statements as necessary.

Let $\empiricaldist := \sum_{i=1}^N \delta_{(\bhx_i, \Xopt_i(\btheta))}$ be the empirical distribution corresponding to the data set, let $\ell(\cdot, \cdot)$ be a loss function penalizing the violation of inverse-feasibility, and consider the inverse problem
\begin{align*}
    \IOPd(\ell, \empiricaldist) := \min_{\btheta \in \bTheta} \;\; \frac{1}{N} \sum_{i=1}^N \ell\left(\bhx_i, \Xopt_i(\btheta) \right).
\end{align*}
This problem is equivalent to Problem~\eqref{eq:data-driven_io_highlevel} introduced in Section~\ref{subsec:problem_def_inverse}, but without the second objective term $h(\btheta)$. While the literature in data-driven inverse optimization typically considers only the inverse-feasibility loss function, the solution algorithms easily adapt when $h(\btheta)$ is introduced. Finally, we note that most of this literature focuses on estimating objective function parameters.

\subsection{Distance from the optimal solution set}\label{sec:datadriven_distance}

An intuitive measure of the error of a forward model with respect to an observed decision $\bhx$ is the distance of $\bhx$ from $\Xopt(\btheta)$. 
We define the Minimum Distance loss function 
\begin{align*}
    \ell_{\mathrm{D}}\left( \bhx, \Xopt(\btheta) \right) := \min_{\bx \in \Xopt(\btheta)} \;\; \norm{\bx -\bhx}_2,
\end{align*}
which measures the $2$-norm distance from the optimal set. 
The data-driven inverse optimization problem $\IOPd(\ell_{\mathrm{D}}, \empiricaldist)$ using the above loss is referred to as the Inverse Distance problem.

\citet{aswani2018inverse} consider convex forward models $\FOPCVX_i(\btheta)$ parametrized by $\bhu_i$ in~\eqref{eq:fop_cvx_u}. They show that for these models, $\IOPd(\ell_\mathrm{D}, \empiricaldist)$ is a statistically consistent estimator of $\btheta$. 
That is, when the data set consists of i.i.d.\ samples from $\field{P}$, the empirical risk minimization problem $\IOPd(\ell_\mathrm{D}, \empiricaldist)$ converges in probability to the expected risk.

\begin{theorem}[\citet{aswani2018inverse}]\label{thm:distance_is_consistent}
    Consider $\FOPCVX_i(\btheta)$ and suppose that
    \begin{enumerate}
        \item[(i)] The feasible sets of $\FOPCVX_i(\btheta)$ are closed, absolutely bounded, and have non-empty interiors for all $\bu \in \set{U}$ and $\btheta \in \bTheta$.

        \item[(ii)] $f_i(\bx, \bu, \btheta)$ and $\bg_i(\bx, \bu, \btheta)$ are continuous in $\bx, \bu, \btheta$, $f(\bx, \bu, \btheta)$ is strictly convex in $\bx$ for fixed $\bu, \btheta$, and $\bg_i(\bx, \bu, \btheta)$ is convex in $\bx$ for fixed $\bu\in\set{U}$ and $\btheta\in\bTheta$.
        
        \item[(iii)] The set $\bTheta$ is closed, bounded, and convex.
        
        \item[(iv)] The data distribution has finite variance $\field{E}_\field{P} [\bx \bx^\tpose] < \boldsymbol{\infty}_{n\times n}$.
    \end{enumerate}
    %
    If for all $\delta > 0$, there exists a sequence of parameter estimates $\btheta_N$ that satisfy
    \begin{align*}
        \lim_{N \rightarrow \infty} \field{P} \left\{ \inf_{\btheta} \left\{ \norm{ \btheta_N - \btheta }_2 \;\bigg|\; \btheta \in \argmin \frac{1}{N} \sum_{i=1}^N \ell_\mathrm{D}(\bhx, \Xopt(\btheta_N)) \right\} > \delta \right\} = 0
    \end{align*}
    then 
    \begin{align*}
        \frac{1}{N} \sum_{i=1}^N \ell_\mathrm{D}(\bhx_i, \Xopt_i(\btheta_N)) \overset{p}{\to} \min_{\btheta \in \bTheta} \;\; \field{E}_\field{P} \left[ \ell_{\mathrm{D}}(\bx, \Xopt(\btheta)) \right].
    \end{align*}
\end{theorem}
Theorem~\ref{thm:distance_is_consistent} states that solving $\IOPd(\ell_{\mathrm{D}}, \empiricaldist)$ yields an estimate $\btheta_N$ that converges to the best estimate $\btheta^*$ that could be obtained with full knowledge of the data distribution $\field{P}$, i.e., the minimizer of $\IOPd(\ell_\mathrm{D}, \field{P})$. This property is known as Risk Consistency.~\citet{aswani2018inverse} also prove a stronger consistency result, Parameter Consistency, where $\btheta_N \overset{p}{\to} \btheta^*$; this property requires additional ``identifiability'' assumptions such as a strictly convex forward objective function.

Although $\IOPd(\ell_{\mathrm{D}}, \empiricaldist)$ is an NP-hard problem for general convex forward models, certain cases permit convex approximations by reformulating $\Xopt(\btheta)$ via KKT conditions (see Section~\ref{sec:classical_IO_convex}) or strong duality~\citep{aswani2018inverse, chan2018inverse, chan2018multiple}.
\begin{property}[Strong Duality for Convex Optimization]
    Assume that $\FOPCVX_i(\btheta)$ has a non-empty interior. Let $h_i(\blambda, \btheta)$ denote the Lagrangian dual function
    \begin{align*}
        h_i(\blambda, \btheta) := \inf_{\bx} \left\{ f_i(\bx, \btheta) + \blambda^\tpose \bg_{i} (\bx, \btheta) \;\Big|\; \bg_i(\bx, \btheta) \leq \bzero \right\}.
    \end{align*}
    Then, $\bx \in \Xopt_i(\btheta)$ if and only if there exists $\blambda \geq \bzero$ for which $f_i(\bx, \btheta) = h_i(\blambda, \btheta)$.
\end{property}

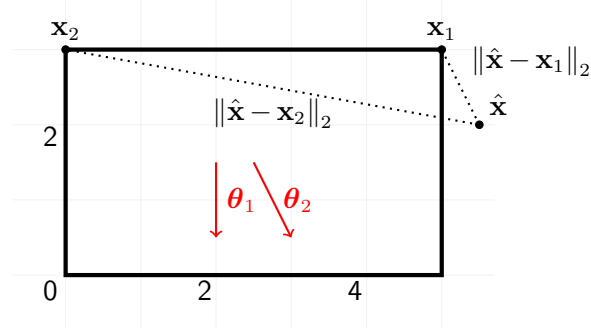
\begin{figure}[t]
\centering
\begin{tikzpicture}[font=\sffamily]
        
    \draw[step=1cm,gray!10,very thin] (-0.7,1.3) grid (5.7,5.7);
    \draw[ultra thick] (0, 5) -- (0, 2) -- (5, 2) -- (5, 5) -- (0, 5);
    \draw[fill=black] (5.5, 4) circle (0.05) node[above right] {$\bhx$};
    \draw[fill=black] (5, 5) circle (0.05) node[above] {$\bx_1$};
    \draw[fill=black] (0, 5) circle (0.05) node[above] {$\bx_2$};
    
    \draw [thick, ->, color=red] (2, 3.5) -- node[right]{$\btheta_1$} (2, 2.5);
    \draw [thick, ->, color=red] (2.5, 3.5) -- node[right]{$\btheta_2$} (3, 2.5);
    
    \draw [thick, dotted] (5.5, 4) -- node[above right]{$\norm{\bhx - \bx_1}_2$} (5, 5);
    \draw [thick, dotted ] (5.5, 4) -- node[below]{$\norm{\bhx - \bx_2}_2$} (0, 5);
    
    \node at (-0.2, 1.80)  {0};
    \node at (-0.2, 3.85)  {2};
    \node at (1.85, 1.80)  {2};
    \node at (3.85, 1.80)  {4};
\end{tikzpicture}
\caption{Visualization of the discontinuity of $\ell_D(\bhx, \Xopt(\btheta))$. The forward model $\FOPL(\btheta)$ is a linear program and $\bhx$ is an infeasible point. For $\btheta_1$, the nearest optimal solution to $\bhx$ is $\bx_1$. However, an infinitesimal rotation  of the estimated cost vector to the left in $\btheta_2$ forces the nearest solution to be $\bx_2$.} 
\label{fig:decision_space_is_discontinuous}
\end{figure}

For most forward models, $\ell_\mathrm{D}(\bx, \Xopt(\btheta))$ is discontinuous in $\btheta$. 
Figure~\ref{fig:decision_space_is_discontinuous} highlights an example. 
Consequently,~\citet{aswani2018inverse} introduce a parameter $\epsilon>0$ and solve the relaxed problem 
%
\begin{align}\label{eq:inverse_risk_minimization_epsilon}
\begin{split}
    \min_{\btheta, \bx_i, \blambda_i} \quad     & \frac{1}{N} \sum_{i=1}^N \norm{\bhx_i - \bx_i}_2 \\
    \st \quad   & f_i(\bx_i, \btheta) - h_i(\blambda_i, \btheta) \leq \epsilon,   \quad \forall i \in \{1, \dots, N\} \\
                & \bg_i(\bx_i, \btheta) \leq \epsilon \bone,   \quad \forall i \in \{1, \dots, N\} \\
                & \blambda_i \geq \bzero,    \quad \forall i \in \{1, \dots, N\} \\
                & \btheta \in \bTheta.
\end{split}
\end{align}
%
Rather than strictly enforcing primal feasibility and strong duality, 
problem~\eqref{eq:inverse_risk_minimization_epsilon} permits these two constraints to be slightly violated. 
The resulting program is convex in $\btheta$ for fixed $\bx_i, \blambda_i$, and in $\bx_i, \blambda_i$ for fixed $\btheta$.
Further,~\citet{aswani2018inverse} develop an enumeration solution algorithm that solves for $\btheta$ and $\bx_i, \blambda_i$ separately. While this approach works for any convex forward model, the authors also show that for strictly convex forward models, there exists an efficient decomposition algorithm that parametrizes $\bx_i$ with Nadaraya-Watson Kernel Regression using the observed decisions.



We present the special case of the enumeration algorithm for $\FOPL(\btheta)$ where we can rewrite $h(\blambda)$ using strong duality. Then, the original $\IOPd(\ell_\mathrm{D}, \empiricaldist)$ becomes a bilinear program
\begin{subequations}\label{eq:inverse_risk_minimization_lp}
\begin{align}
    \min_{\btheta, \bx_i, \blambda_i} \quad & \frac{1}{N} \sum_{i=1}^N \norm{\bhx_i - \bx_i}_p \\
    \st \quad   & \bA_i^\tpose \blambda_i = \btheta, \quad \blambda_i \geq \bzero,  \quad \forall i \in \{1, \dots, N\} \label{eq:inverse_risk_minimization_lp1} \\
                & \btheta^\tpose \bx_i = \bb_i^\tpose \blambda_i \quad\forall i \in \{1, \dots, N\} \label{eq:inverse_risk_minimization_lp3} \\
                & \bA_i \bx_i \geq \bb_i, \quad \forall i \in \{1, \dots, N\} \label{eq:inverse_risk_minimization_lp4} \\
                & \btheta \in \bTheta.
\end{align}
\end{subequations}
Analogous to~\eqref{eq:inverse_risk_minimization_epsilon}, the relaxed version of problem~\eqref{eq:inverse_risk_minimization_lp} replaces constraints~\eqref{eq:inverse_risk_minimization_lp3} and~\eqref{eq:inverse_risk_minimization_lp4} with $\btheta^\tpose \bx_i \leq \bb_i^\tpose \blambda_i + \epsilon$ and $\bA_i \bx_i \geq \bb_i - \epsilon \bone$, respectively.
In the enumeration algorithm, we discretize $\bTheta$ by a $\delta$-net, i.e., a finite set $\set{N}(\delta) \subset \bTheta$ that satisfies $\max_{\btheta \in \bTheta} \min_{\bhtheta \in \set{N}(\delta)} \|\btheta - \bhtheta\| \leq \delta$. We then enumerate over $\hat\btheta \in \set{N}(\delta)$, solve the problem for each $\hat\btheta$, and select the best parameter.

\citet{chan2018inverse} and~\citet{chan2018multiple} explore a special case of $\IOPd(\ell_{\mathrm{D}}, \empiricaldist)$ where all of the points in $\empiricaldist$ correspond to solutions for the same instance of the forward model, i.e., $\FOPL_i(\btheta) = \FOPL(\btheta)$ for all $i$. This enforces a single dual variable, i.e., $\blambda_i = \blambda$ for all $i$. 
Here, the dual feasibility constraint~\eqref{eq:inverse_risk_minimization_lp1} implies that the cost parameter lies in the row space of the constraint matrix $\btheta = \bA^\tpose \blambda$ where $\blambda \in \field{R}^m_+$. Assuming that $\bTheta$ is sufficiently large but excludes $\bzero$ (i.e., preventing trivial solutions), there must exist a $\blambda^*$ in the extreme rays of $\field{R}^m_+$ for which the corresponding $\btheta^*$ is optimal. This means that there exists $j$ where $\btheta^*$ is equal to $\ba_j$ multiplied by a normalization factor.
Furthermore,~\citet{chan2018inverse} show that when $N=1$, the inverse optimization problem has a closed-form solution.
In this special case of the forward model, the number of enumerations reduces from $|\set{N}(\delta)|$ (which depends on the size of $\bTheta$) to $m$.

The geometric property equating optimal parameters to rows of the constraint matrix $\bA$ may lead to instability in estimates~\citep{shahmoradi2021quantile, gupta2022decomposition, ahmadi2020inverse}. For example if the rows of $\bA$ correspond to vectors that point in orthogonal directions and the estimated $\btheta$ must lie within this discrete set, then seemingly small variations in the data set may yield large changes in the estimate. Consequently,~\citet{gupta2022decomposition} propose restricting the distance minimization loss to only consider vertices of the feasible set. They propose the alternative loss function $\min_{\bx} \{ \norm{\bx -\bhx}_2 \;|\; \bx \in \Xopt(\btheta) \cap \extreme(\{\bx \;|\; \bA \bx \geq \bb \}) \}$ and develop a two-phase mixed integer linear programming reformulation for the corresponding inverse problem.


\subsection{Sub-optimality in the objective function value}\label{subsec:suboptimality}

Another intuitive and popular loss function to minimize when estimating $\btheta$ is the degree of sub-optimality of the observed decisions $\bhx_i$ under the estimated models $\FOP_i(\btheta)$. 
We can evaluate sub-optimality with two potential loss functions.
\begin{itemize}
    \item \textbf{Absolute Sub-optimality:} This loss function measures the difference in the objective function values of the observed decisions with respect to the estimated optimal values, i.e., 
    \begin{align*}
        \ell_\mathrm{ASO}\left(\bhx, \Xopt(\btheta)\right) := \left| f(\bhx, \btheta) - \min_{\bx \in \Xfeas(\btheta)} f(\bx, \btheta) \right|.
    \end{align*}
    
    \item \textbf{Relative Sub-optimality:} This loss function measures the competitive ratio of the objective function values of the observed decisions with respect to the estimated optimal values, i.e.,
    \begin{align*}
        \ell_\mathrm{RSO}\left(\bhx, \Xopt(\btheta)\right) := \left| \frac{f(\bhx, \btheta)}{\min_{\bx \in \Xfeas(\btheta)} f(\bx, \btheta)} - 1 \right|.
    \end{align*}
    
\end{itemize}

Let $\IOPd(\ell_\mathrm{ASO}, \field{P}_N)$ and $\IOPd(\ell_\mathrm{RSO}, \empiricaldist)$ denote the Inverse Absolute Sub-optimality and Inverse Relative Sub-optimality problems, respectively. 
While these inverse problems do not possess statistical properties, they often lead to tractable optimization problems, especially for linear forward models. Moreover, $\IOPd(\ell_\mathrm{ASO}, \empiricaldist)$ yields generalization bounds (see Section~\ref{sec:datadriven_vi}).

\subsubsection{Absolute sub-optimality loss.}\label{subsubsec:abs_suboptimality}

We first consider a special case where the observed decisions are known to be feasible, i.e., $\bhx_i \in \Xfeas_i(\btheta)$ for all $(\bhx_i, \FOP_i(\btheta)) \in \dataset$ and $\btheta \in \bTheta$. Here, we can remove the absolute values in the loss function, reformulating the inverse problem to
\begin{align*}
    \min_{\btheta, \bx_i} \quad & \frac{1}{N} \sum_{i=1}^N (f_i(\bhx_i, \btheta) - f_i(\bx_i, \btheta)) \\
    \st \quad   & \bx_i \in \Xopt_{i}(\btheta), \quad \forall i \in \{1, \dots, N\} \\
                & \btheta \in \bTheta.
\end{align*}
This problem can be solved using difference of convex programming techniques.

For linear forward models, the general (i.e., without assuming feasible observations) Inverse Absolute Sub-optimality problem can be reformulated and efficiently solved using strong duality.~\citet{chan2014generalized, chan2018inverse} explore this problem for a single observation.~\citet{chan2018multiple} extend these results to multiple observed decisions. We provide a slight generalization of their result below. 
\begin{theorem}\label{thm:aso_general}
    Consider $\FOPL_i(\btheta)$. Then, $\IOPd(\ell_\mathrm{ASO}, \empiricaldist)$ is equivalent to the linear program 
    \begin{align}\label{eq:aso_lp_form}
    \begin{split}
        \min_{\btheta, \blambda_i} \quad & \frac{1}{N} \sum_{i=1}^N | \btheta^\tpose \bhx_i - \bb_i^\tpose \blambda_i | \\
        \st \quad   & \bA_i^\tpose \blambda_i = \btheta, \quad \blambda_i \geq \bzero, \quad \forall i \in \{1, \dots, N\} \\
                    & \btheta \in \bTheta.
    \end{split}
    \end{align}
\end{theorem}
Problem~\eqref{eq:aso_lp_form} generalizes the Inverse Linear Optimization problem of~\citet{ahuja2001inverse} from Section~\ref{subsec:linear}. If $\bhx_i$ are all optimal for some $\btheta$, then the optimal objective value will be 0 and we can replace the objective with constraints $\btheta^\tpose \bhx_i = \bb_i^\tpose \blambda$ for all $i$, recovering the original problem~\eqref{eq:ahuja_ilo}.

The number of constraints and variables in problem~\eqref{eq:aso_lp_form} scale with the number of data points. If the instances of the forward models are large optimization problems, then the inverse problem can quickly become difficult to solve. 
However, similar to the Minimum Distance loss setting,~\citet{chan2018multiple} show that when the forward model is the same for all of the observed decisions, i.e., $\FOPL_i(\btheta) = \FOPL(\btheta)$ for all $i$, then problem \eqref{eq:aso_lp_form} can be rewritten as a smaller linear program
\begin{align}\label{eq:aso_lp_form_one_feas_region}
    \begin{split}
        \min_{\btheta, \blambda} \quad & \frac{1}{N} \sum_{i=1}^N | \btheta^\tpose \bhx_i - \bb^\tpose \blambda | \\
        \st \quad   & \bA^\tpose \blambda = \btheta, \quad \blambda \geq \bzero  \\
                    & \btheta \in \bTheta.
    \end{split}
\end{align}
The key observation is that when the forward model is assumed to be the same across instances, we only need to control a single dual variable rather than $N$ of them.

The Absolute Sub-optimality loss can also be used to estimate parameters in the constraints of linear forward models.
\citet{chan2020inverse} and~\citet{ghobadi2021inferring} generalize their classical methods for constraint estimation (see Section~\ref{sec:classical_io_linear_estimating_constraints}) to the data-driven setting. Recall that~\citet{chan2020inverse} consider instances of forward optimization models $\FOPL_i(\bPhi)$ where the cost and right-hand-side constraint vectors $\bc_i$ and $\bb_i$, respectively, are given. In the data-driven setting, they formulate the problem
\begin{align*}
    \min_{\bPhi, \blambda_i} \quad  & \frac{1}{N} \sum_{i=1}^N | \bc_i^\tpose \bhx_i - \bb_i^\tpose \blambda_i | \\
    \st \quad   & \bPhi^\tpose \blambda_i = \btheta,  \quad \blambda_i \geq \bzero,  \quad\forall i \in \{1, \dots, N\} \\
                & \bPhi \bhx_i \geq \bb_i, \quad \forall i \in \{1, \dots, N\} \\
                & \bPhi \in \bTheta.
\end{align*}
This problem is a bilinear program, but~\citet{chan2020inverse} show that when the forward models are identical across instances, under a well-behaved $\bTheta$, the problem has an analytic solution.

%
%

\subsubsection{Relative sub-optimality loss.}
\label{subsubsec:data-driven_IO_rdg}

Minimizing the Relative Sub-optimality loss function leads to a fractional programming problem, which are generally difficult to solve \citep{frenk2005fractional}. However, this fractional component can be removed if the forward model satisfies a scaling invariance property. Specifically, if we can scale $\btheta$ while preserving the same optimal value and solution set, then the denominator in the loss function becomes irrelevant up to a scaling factor. 
\begin{definition}[Scaling invariance]
    A forward model $\FOP(\btheta)$ is invariant to scaling if $\bx^* \in \Xopt(\btheta)$ implies that (i) $\bx^* \in \Xopt(\alpha \btheta)$ and (ii) $f(\bx^*, \alpha \btheta) = \alpha f(\bx^*, \btheta)$ for any $\alpha > 0$.
\end{definition}
If the forward model satisfies this property and has a non-negative optimal value, then we can always scale the parameter while preserving the solution set. Here, $\IOPd(\ell_{\mathrm{RSO}}, \empiricaldist)$ becomes
\begin{align*}
    \min_{\btheta, \bx_i} \quad  & \frac{1}{N} \sum_{i=1}^N | f_i(\bhx_i, \btheta) - 1 | \\
    \st \quad   & f_i(\bx_i, \btheta) = 1, \quad \forall i \in \{1, \dots, N\} \\
                & \bx_i \in \Xopt_i(\btheta), \quad \forall i \in \{1, \dots, N\} \\
                & \btheta \in \bTheta
\end{align*}
Forward models with basis-function objectives $f(\bx, \btheta) = \sum_{b=1}^B \theta_b f^{(b)}(\bx)$ are scaling invariant. Note that setting $f^{(b)} = x_b$ for all $b$ yields a linear objective, meaning $\FOPL_i(\btheta)$ is also scaling invariant.

\citet{chan2014generalized, chan2018inverse} minimize the Relative Sub-optimality loss for a single decision $(\bhx, \FOPL(\btheta))$. \citet{chan2018multiple} extend these results to the case of multiple observed decisions. In both cases, they characterize $\Xopt(\btheta)$ using strong duality. We provide below a revised version of their results for multiple decisions and instances.
\begin{theorem}\label{thm:io_rdg}
    Consider $\FOPL_{i}(\btheta)$. If $\bb_i > \bzero$ (component-wise) for all $i \in \{1, \dots, N\}$, then $\IOPd(\ell_{\mathrm{RSO}}, \empiricaldist)$ is equivalent to the following linear program 
    \begin{align}\label{eq:rso_lp_form}
    \begin{split}
        \min_{\btheta,\blambda_i} \quad & \frac{1}{N} \sum_{i=1}^N | \btheta^\tpose \bhx_i - 1 | \\
        \st \quad   & \bb_i^\tpose \blambda_i = 1, \quad \forall i \in \{1, \dots, N\} \\
                    & \bA_i^\tpose \blambda_i = \btheta, \quad \blambda_i \geq \bzero, \quad \forall i \in \{1, \dots, N\} \\
                    & \btheta \in \bTheta
    \end{split}
    \end{align}
\end{theorem}
Similar to the Absolute Sub-optimality case, when the forward model is the same across instances, the inverse problem can be drastically reduced in size~\citep{chan2018multiple}. Furthermore, the positivity assumption $\bb_i > \bzero$ is not necessary for solving $\IOPd(\ell_{\mathrm{RSO}}, \empiricaldist)$; we refer to~\citet{chan2018multiple} for the generalization. 

\citet{troutt2005linear, ref:troutt_ejor08} explore the Relative Sub-optimality loss
for the problem of jointly estimating the objective parameters and constraint matrix for $\FOPL((\btheta, \bPhi)) := \min \{ \btheta^\tpose \bx \;|\; \bPhi \bx \geq \bb \}$. 
Using scaling invariance, they reformulate their inverse problem to the one in Theorem~\ref{thm:io_rdg}, except now replacing $\bA_i$ with variables $\bPhi_i$. The resulting inverse problem is non-convex, but~\citet{troutt2005linear, ref:troutt_ejor08} propose an enumeration algorithm by discretizing the set of potential $\bPhi$.

\subsection{Variational inequality loss} \label{sec:datadriven_vi}

The first-order Variational Inequality (VI) is an optimality criterion for any general convex optimization problem with a differentiable objective function. 
\begin{property}[First-order VI]
    The solution $\bhx_i$ is optimal for $\FOPCVX_i(\btheta)$ if and only if
    \begin{align*}
        \nabla_{\bx} f_i(\bhx_i, \btheta)^\tpose \left(\bhx_i - \bx \right) \leq 0, \quad \forall \bx \in \Xfeas_i(\btheta).
    \end{align*}
\end{property}
\citet{bertsimas2015data} use the violation of this variational inequality as a loss measure for data-driven inverse optimization. They propose the following Variational Inequality loss function: 
\begin{align*}
    \ell_{\mathrm{VI}}\left( \bhx, \Xopt(\btheta) \right) := \max_{\bx \in \Xfeas(\btheta)} \;\;  \nabla_{\bx} f(\bhx, \btheta)^\tpose\left( \bhx - \bx \right).
\end{align*}
Let $\IOPd(\ell_\mathrm{VI}, \empiricaldist)$ denote the Inverse Variational Inequality problem. This problem is similar to the Sub-optimality problems in that it can usually be solved using convex programming. However, the inverse problem does not immediately yield statistical consistency, meaning there is no guarantee that the empirical risk problem $\IOPd(\ell_\mathrm{VI}, \empiricaldist)$ converges in probability to the expected risk problem $\IOPd(\ell_\mathrm{VI}, \field{P})$. Instead,~\citet{bertsimas2015data} develop a generalization bound, i.e., an upper bound on the expected risk as a function of the empirical risk. This bound holds specifically for when the forward models are conic optimization problems.
\begin{theorem}[\citet{bertsimas2015data}]\label{thm:ivi_generalization}
    Consider the Conic Forward Optimization Model 
    \begin{align*}
        \FOPCSF_i(\btheta) := \min_{\bx} \left\{ f_i(\bx, \btheta) \;|\; \bA_i \bx = \bb_i, \; \bx \in \set{C}_i \right\},
    \end{align*}
    where $\set{C}_i$ are convex cones. Suppose for $(\bx, \FOPCSF(\btheta)) \sim \field{P}$ that
    \begin{enumerate}
        \item[(i)] $\bA \bx = \bb, \bx \in \set{C}$ almost surely.
        
        \item[(ii)] There exists $\btx \in \interior(\set{C})$ such that $\bA \btx = \bb$ almost surely.
        
        \item[(iii)] The feasible sets are absolutely bounded in a ball of radius $R$.
        
    \end{enumerate}
    Let $\bar{B} := 2 \sup \{ f(\bx, \btheta) \;|\; \norm{\bx}_2 \leq R, \; \btheta \in \bTheta \}$. Then with probability at least $1-\beta$, 
    \begin{align*}
        \field{E}_\field{P} \left[ \ell_\mathrm{VI} \left( \bx, \Xopt(\btheta) \right) \right] \leq \frac{1}{N} \sum_{i=1}^N \ell_\mathrm{VI}\left(\bhx_i, \Xopt_i(\btheta)\right) + \frac{1}{\sqrt{N}} \left( 4 \bar{B} + \sqrt{8\bar{B} \log(2/\beta)} \right) \quad \forall \btheta \in \bTheta.
    \end{align*}
\end{theorem}
Theorem~\ref{thm:ivi_generalization} is a standard type of generalization bound used to evaluate machine learning models~\citep{bartlett2002rademacher}. For any estimated parameter, the theorem bounds the expected risk of variational inequality violation as a function of the empirical risk as observed by the data set and a constant factor that is a function of the number of points $N$ and the size of the feasible set $\bar{B}$. This constant factor scales as $O(1/\sqrt{N})$: with larger data sets of decisions, the empirical risk provides an increasingly accurate approximation of the expected risk. 

We first remark that for linear forward models, the loss functions for Inverse Variational Inequality and Inverse Absolute Sub-optimality are equivalent:
\begin{align*}
        \ell_{\mathrm{VI}}\left( \bhx, \Xopt(\btheta) \right) &= \max_{\bx \in \Xfeas} \;   \btheta^\tpose\left( \bhx - \bx \right) 
        = \btheta^\tpose \bhx - \min_{\bx \in \Xfeas} \btheta^\tpose \bx 
        = \left|\btheta^\tpose \bhx - \min_{\bx \in \Xfeas} \btheta^\tpose \bx \right| = \ell_\mathrm{ASO}(\bhx, \Xopt(\btheta))
\end{align*}
%
Recall that $\FOPL(\btheta)$ is a special case of $\FOPCSF(\btheta)$ by setting $\set{C} = \field{R}^+$. 
Then, the generalization bound in Theorem~\ref{thm:ivi_generalization} also applies for Inverse Absolute Sub-optimality with linear forward models. 
Furthermore, since the Inverse Absolute Sub-optimality problem possesses efficient solution algorithms, the Inverse Variational Inequality problem can be easily solved. 

In general, computing $\ell_\mathrm{VI}(\bhx, \Xopt(\btheta))$ requires maximizing a convex optimization problem, and thus, minimizing $\IOPd(\ell_\mathrm{VI}, \empiricaldist)$ outright requires solving a minimax optimization problem. However,~\citet{bertsimas2015data} observe that conic optimization problems admit a convenient dual form~\citep{aghassi2006solving}. This yields a convex reformulation. 
\begin{theorem}[\citet{bertsimas2015data}]\label{thm:ivi_duality_solution}
    If for each $(\bhx_i, \FOPCSF_i(\btheta))$, there exists $\btx \in \interior(\set{C}_i)$ such that $\bA_i\btx = \bb_i$, then $\IOPd(\ell_\mathrm{VI}, \empiricaldist)$ is equivalent to the following problem 
    \begin{align*}
        \min_{\bepsilon, \btheta, \blambda} \quad  & \sum_{i=1}^N \left| \nabla_{\bx} f_i(\bhx_i, \btheta)^\tpose \bhx_i - \bb_i^\tpose \blambda_i \right| \\
    \st\quad    & \nabla_{\bx} f_i(\bhx_i, \btheta) - \bA_{i}^\tpose \blambda_i \in \set{C}_i   \quad \forall i \in \{1, \dots, N\} \\
                & \btheta \in \bTheta.
    \end{align*}
    Furthermore if $\nabla_{\bx} f_i(\bx, \btheta)$ is a linear function of $\btheta$, then the above problem is convex. 
\end{theorem}

Theorem~\ref{thm:ivi_duality_solution} shows that the Inverse Variational Inequality problem can be solved via convex programming for forward optimization models with conic constraints.
This theorem holds for any of the forward models introduced in Section~\ref{sec:problem_def}. 
The Inverse Variational Inequality problem is attractive as it can be efficiently solved via convex programming for a large class of convex forward models, while also possessing theoretical guarantees on the quality of the inverse solution. 

\subsection{Violating the KKT conditions}

\citet{keshavarz2011imputing} consider $\FOPCVX(\btheta)$ and propose loss functions that describe the degree to which each observed decision violates the KKT conditions. 
Recall that for a decision to be optimal for $\FOPCVX(\btheta)$ there must exist a dual variable such that the decision and dual satisfy a stationarity and a complementary slackness condition. We define two loss terms below that measure the degree to which a primal and dual solution pair $(\bx, \blambda)$ violate these conditions:
\begin{align*}
    \ell_\mathrm{st}(\bx, \Xopt(\btheta), \blambda) &:= \left| \nabla_{\bx} f(\bx, \btheta) + \sum_{j=1}^m \lambda_{j} \nabla_{\bx} g_{j}(\bx, \btheta) \right| \\
    \ell_\mathrm{cs}(\bx, \Xopt(\btheta), \blambda) &:= \norm{ \left(\lambda_1 g_1(\bx, \btheta), \lambda_2 g_2(\bx, \btheta), \ldots, \lambda_m g_m(\bx, \btheta) \right) }_p.
\end{align*} 
Given $(\bhx, \Xopt(\btheta))$, the KKT loss function finds a dual vector to minimize the sum of these residuals 
\begin{align*}
    \ell_{\mathrm{KKT}}\left(\bhx, \Xopt(\btheta) \right) := \min_{\blambda \geq \bzero} \;\; \Big (  \ell_\mathrm{st}\left(\bhx, \Xopt(\btheta), \blambda\right) + \ell_\mathrm{cs}\left(\bhx, \Xopt(\btheta), \blambda \right) \Big).
\end{align*}

Let $\IOPd(\ell_\mathrm{KKT}, \empiricaldist)$ be the Inverse KKT problem. For a given $\bhx$, computing $\ell_\mathrm{KKT}(\bhx, \Xopt(\btheta))$ requires solving a convex optimization problem. Consequently, $\IOPd(\ell_\mathrm{KKT}, \empiricaldist)$ becomes a non-convex problem in $(\btheta, \blambda)$ due to $\ell_\mathrm{cs}(\bhx, \Xopt(\btheta), \blambda)$. However, if the goal is to estimate only the parameters of the objective, then the inverse problem can be solved with convex programming.
\begin{proposition}
    Consider $\FOPCVX(\btheta)$. If $\bg_i(\bx, \btheta) = \bg_i(\bx)$ and $f(\bx, \btheta)$ is convex in $\btheta$, then $\IOPd(\ell_\mathrm{KKT}, \empiricaldist)$ is a convex optimization problem.
\end{proposition}

\subsection{Alternatives to empirical risk minimization}

So far, we have explored inverse optimization methods that minimize the expectation of a loss function over an empirical data distribution. We may also minimize alternative risk measures. For any probability distribution $\field{P}$, let $\rho^\field{P}(\cdot)$ denote a general risk measure~\citep{ruszczynski2006optimization, rockafellar2002conditional}, and consider 
\begin{align*}
    \IOPdr(\ell, \field{P}, \rho) := \min_{\btheta \in \bTheta} \;\; \rho^\field{P}\left( \ell\left( \bx, \Xopt(\btheta) \right) \right). 
\end{align*}
Setting $\rho^\field{P}(\cdot) = \field{E}_\field{P}[\cdot]$ recovers all of the previous models discussed in this section. 

\subsubsection{Quantile inverse optimization.}

Empirical risk minimization yields estimators that may be sensitive to outliers~\citep{hastie2009elements}. 
Ideally, a small perturbation in $\field{P}_N$ should not result in a large perturbation of the estimated parameter $\btheta^*$. 
For example in linear regression, squared-error loss penalizes larger errors and invites sensitivity, motivating alternative approaches such as quantile regression \citep{koenker2001quantile}. 
To reduce sensitivity in inverse optimization,~\citet{shahmoradi2021quantile} propose the Value-at-Risk (VaR) or $\chi$-quantile risk function
\begin{align*}
    \VaR_\chi(\ell, \field{P}) &:= \inf_{\tau \geq 0} \left\{ \tau \;\Big|\; \field{P}\left\{ \ell(\bx, \Xopt(\btheta)) \leq \tau \right\} \geq \chi \right\},  
\end{align*}
where $\chi \in [0, 1]$ is a percentile parameter that we choose. Depending on the predetermined value of $\chi$, the Inverse VaR problem $\IOPdr(\ell, \empiricaldist, \VaR)$ yields parameter estimates $\btheta^*$ for which the error has a bounded sensitivity with respect to perturbations in the data set.

Directly minimizing the $\chi$-quantile risk in $\IOPdr(\ell, \empiricaldist, \VaR)$ is a challenging optimization problem since the Value-at-Risk is discrete. Given a data set, the Value-at-Risk is defined with indicator functions, $\inf_\tau \{ \tau \;|\; \sum_{i=1}^N \Ind \{ \ell(\bhx_i, \Xopt_i(\btheta)) \leq \tau \} \geq \chi \}$. 
However, if we assume that given $\FOP_i(\btheta)$, the loss is bounded above by some value $M$, then the Inverse VaR problem becomes 
\begin{align*}
    \min_{\btheta, \tau, \bpi} \quad & \tau \\
    \st\quad    & \ell \left( \bhx_i, \Xopt_i(\btheta) \right) \leq \tau + M (1 - \pi_i), \quad \forall i \in \{1, \ldots, N\} \\
                & \sum_{i=1}^N \pi_i \geq N\cdot\chi  \\
                & \btheta \in \bTheta, \quad \bpi \in \{0, 1\}^N.
\end{align*}

The difficulty in solving the Inverse VaR problem depends on the choice of $\ell(\bx, \Xopt(\btheta))$.~\citet{shahmoradi2021quantile} explore linear forward models $\FOPL(\btheta)$ and the Minimum Distance loss $\ell_\mathrm{D}(\bx, \Xopt(\btheta))$ for general $p$-norms. Here, the resulting problem is a mixed integer conic program
\begin{align}\label{eq:inverse_quantile_micp}
\begin{split}
    \min_{\btheta, \bx_i, \blambda_i, \bpi_i} \quad & \tau \\
    \st \qquad   & \bA_i^\tpose \blambda_i = \btheta, \quad \blambda_i \geq \bzero , \quad \forall i \in \{1, \dots, N\} \\
                & \btheta^\tpose \bx_i = \bb_i^\tpose \blambda_i , \quad \forall i \in \{1, \dots, N\} \\
                & \bA_i \bx_i \geq \bb_i,  \quad \forall i \in \{1, \dots, N\} \\
                & \norm{\bhx_i - \bx_i}_p \leq \tau + M (1-\pi_i) , \quad \forall i \in \{1, \dots, N\} \\
                & \sum_{i=1}^N \pi_i \geq N \cdot \chi  \\
                & \btheta \in \bTheta, \quad \bpi \in \{0, 1\}^N,
\end{split}
\end{align}
which becomes a mixed integer linear program for $p=1$ or $p=\infty$.

Rather than directly solving Problem~\eqref{eq:inverse_quantile_micp},~\citet{shahmoradi2021quantile} fix a target $\chi$-quantile loss $\tau = \hat\tau$ and search for a feasible $\btheta \in \bTheta$ that can achieve this target. Although the original Problem~\eqref{eq:inverse_quantile_micp} is NP-hard, fixing $\tau$ can recast the problem to a bi-clique problem for which there exists effective solution algorithms and heuristics~\citep{dawande2001bipartite}. Second, fixing the Inverse VaR loss allows for introducing additional alternative objective functions for the inverse problem. In particular,~\citet{shahmoradi2021quantile} explore several statistical properties of forward model stability that they show can be satisfied by fixing $\tau$ and minimizing an alternate objective.

\subsubsection{Distributionally robust inverse optimization.}

Generalizing $\IOPdr(\ell, \field{P}, \rho)$, \citet{esfahani2018data} propose a distributionally robust inverse optimization problem where the ambiguity set controls a worst-case data distribution from the empirical $\empiricaldist$.

For any two distributions $\field{P}$ and $\field{Q}$, let
\begin{align*}
    W(\field{P}, \field{Q}) := \inf_{\bPi} \left\{ \field{E}_{\bPi} \left[ \norm{\bx - \bx'}_2 \right] \;\Bigg|\;
    \begin{array}{ll} \displaystyle
        \text{$\bPi$ is a joint distribution of $\bx, \bx'$} \\
        \text{with marginals $\field{P}$ and $\field{Q}$, respectively}
    \end{array}
    \right\}
\end{align*}
be the Wasserstein distance between them. This distance is well-defined for both continuous and discrete distributions, meaning we can compute Wasserstein distances with respect to $\empiricaldist$. 
Consider a ball-uncertainty set $\set{B}(\empiricaldist, \epsilon) := \{ \field{Q} \;|\; W(\field{Q}, \empiricaldist) \leq \epsilon \}$ of size $\epsilon > 0$ around the empirical distribution and define the Distributionally Robust Inverse problem
\begin{align*}
    \IOPddro(\ell, \empiricaldist, \rho, \epsilon) := \min_{\btheta} \sup_{\field{Q}} \quad & \rho^\field{Q}\left(\ell\left(\bx, \Xopt(\btheta)\right)\right) \\
    \st \quad   & \field{Q} \in \set{B}(\empiricaldist, \epsilon) \\
                & \btheta \in \bTheta.
\end{align*}

This problem naturally provides an out-of-sample guarantee on the true risk. 
%
\begin{proposition}[\citet{esfahani2018data}]\label{propn:dro_io_generalization}
    Consider $\FOPCVX_i(\btheta)$ instances that are defined by known input parameters $\bhu_i$. Assume that there exists $\alpha > 1$ for which $A_\alpha := \field{E}_\field{P}[\exp(\norm{(\bx, \bu)}^\alpha)] < \infty$. Let $\btheta^*$ be an optimal parameter estimate obtained by solving $\IOPddro(\ell, \empiricaldist, \rho, \epsilon)$ and let $z^*_\mathrm{DRO}$ be the optimal value. Then, there exists $\beta(\epsilon, A_\alpha, N)$ for which
    \begin{align*}
        \field{P} \left\{ \rho^\field{P}\left(\ell\left(\bx, \Xopt(\btheta)\right)\right) \leq z^*_\mathrm{DRO} \right\} \geq 1 - \beta(\epsilon, A_\alpha, N).
    \end{align*}
\end{proposition}
We refer to~\citet{esfahani2018data} for details on computing $\beta$. Note that when $\rho^\field{P}(\cdot) = \field{E}_\field{P}[\cdot]$, the out-of-sample guarantee is similar to a generalization bound (see Theorem~\ref{thm:ivi_generalization}) in terms of bounding the true risk. 
However, Proposition~\ref{propn:dro_io_generalization} holds independent of the risk measure or the loss function used to evaluate the inverse problem, meaning that the distributionally robust framework applies to any of the inverse optimization variants that were previously introduced. 
Furthermore, this guarantee directly compares the empirical and true risk with no additional terms.

\citet{esfahani2018data} explore the Conditional-Value-at-Risk (CVaR) at level $\alpha > 0$ as their risk measure for distributionally robust inverse optimization
\begin{align*}
    \CVaR(\ell, \alpha, \field{P}) &:=  \inf_\tau \tau + \frac{1}{\alpha} \field{E}_\field{P} \left[ \max\left\{ \ell(\bx, \Xopt(\btheta)) - \tau , \; 0 \right\} \right]. 
\end{align*}
CVaR generalizes the empirical risk as a special case when $\alpha = 1$ and further generalizes the essential supremum risk measure for $\alpha \downarrow 0$. 
However, the distributionally robust inverse problem $\IOPddro(\ell, \field{P}_N, \CVaR, \epsilon)$ is NP-hard for many forward model classes. For linear forward models $\FOPL_i(\btheta)$ and the sub-optimality loss function $\ell_\mathrm{ASO}(\bx, \Xopt(\btheta))$, \citet{esfahani2018data} use techniques from distributionally robust optimization~\citep{esfahani2018dro} to show that $\IOPddro(\ell_\mathrm{ASO}, \field{P}_N, \text{CVaR}, \epsilon)$ can be re-formulated as a large conic optimization problem.

\cite{dong2020wasserstein} build on the distributionally robust framework of~\citet{esfahani2018data} for multi-objective forward optimization models
\begin{align*}
    \FOPCVX\text{--}\prob{MO}((\btheta, \bphi)) := \min_{\bx} \left\{ \sum_{b=1}^{B} \phi_b f^{(b)}(\bx, \btheta) \;\Bigg|\; \bg(\bx, \btheta) \leq \bzero \right\}.
\end{align*}
$\FOPCVX\text{--}\prob{MO}((\btheta, \bphi))$ reduces to the conventional convex forward model $\FOPCVX(\btheta, \bphi)$ when considering only a single objective. \citet{dong2020wasserstein} use the expected value risk function $\rho^\field{P}(\cdot) = \field{E}_\field{P}[\cdot]$ and distance minimization loss $\ell_\mathrm{D}$, meaning that their inverse problem is a multi-objective distributionally robust generalization of the Inverse Distance problem of~\citet{aswani2018inverse}. Since the inverse problem of~\citet{aswani2018inverse} itself is NP-hard, the distributionally robust formulation is even more difficult to solve. Using duality,~\citet{dong2020wasserstein} show that the distributionally robust problem can be re-formulated into a semi-infinite optimization problem under certain conditions, which can be solved using constraint generation.

\subsection{Inverse optimization with non-parametric forward models}\label{subsec:non-param}

\citet{bertsimas2015data} consider non-parametric kernel functions to model a forward optimization objective.
A kernel $K: \field{R}^n \times \field{R}^n \rightarrow \field{R}$ is a similarity function that can compare a given point $\bx$ with another $\bhx_i$ in the data set (see~\citet{scholkopf2001learning} and~\citet{hastie2009elements} for reviews from a machine learning perspective).
In inverse optimization, a non-parametric objective can be modeled indirectly by representing the gradient $\nabla_{\bx} f(\bx)$ as a linear combination of kernels with respect to the observed decisions:
\begin{align}\label{eq:kernel_forward_objective}
    \nabla_{\bx} f(\bx) = \btheta \bK(\bx) := 
    \begin{bmatrix}
        \theta_{1,1} & \theta_{1, 2} & \cdots & \theta_{1, N} \\
        \theta_{2,1} & \theta_{2, 2} & \cdots & \theta_{2, N} \\
        \vdots     & \vdots      & \ddots & \vdots \\
        \theta_{n,1} & \theta_{n, 2} & \cdots & \theta_{n, N}
    \end{bmatrix}
    \begin{bmatrix}
        K(\bhx_1, \bx) \\
        K(\bhx_2, \bx) \\
        \vdots \\
        K(\bhx_N, \bx)
    \end{bmatrix}.
\end{align}
Non-parametric representations do not permit a closed form of objective $f(\bx)$ from~\eqref{eq:kernel_forward_objective} and can reduce forward model interpretability at the trade-off of flexibility by being able to characterize the gradient of an arbitrary convex function. Thus, kernel-based forward models are valuable when we do not have a clear structural understanding of an agent's decision-making process.

While the inverse optimization literature has primarily explored kernel representations of forward optimization objectives for Inverse Variational Inequality problems, kernel representations can be used in conjunction with any loss functions that depend on computing gradients of a forward objective. Furthermore, note that the kernel representation is linear in the kernel weights, meaning that the use of a kernel function typically will not increase the computational complexity of the inverse optimization problem. Finally, inverse optimization problems for estimating constraints often require gradient representations of the constraint functions (e.g., Inverse KKT). Consequently, non-parametric techniques can also be used to model constraints in such settings.


Depending on the choice of kernel function, $\btheta \bK(\bx)$ can represent function spaces that are arbitrarily complex~\citep{hastie2009elements}, meaning that this non-parametric technique may overfit to the training decision data.~\citet{bertsimas2015data} recommend including a regularization penalty in the inverse optimization objective. To demonstrate, we consider $\FOPCSF(\btheta)$ and provide the non-parametric formulation of $\IOPd(\ell_{\mathrm{VI}}, \empiricaldist)$ below: 
\begin{align*}
    \min_{\bepsilon, \btheta, \blambda} \quad  & \sum_{i=1}^N \norm{\epsilon^{\mathrm{vi}}_i}_1 +  \kappa \sum_{k=1}^n \sum_{i=1}^N \sum_{i'=1}^N \theta_{k, i} \theta_{k, i'} K(\bhx_i, \bhx_{i'})    \\
    \st\quad    & \epsilon^{\mathrm{vi}}_i \geq \left( \btheta \bK(\bhx_i) \right)^\tpose \bhx_i - \bb_i^\tpose \blambda_i,   \quad \forall i \in \{1, \dots, N\}  \\
                &  \btheta \bK(\bhx_i) - \bA_{i}^\tpose \blambda_i \in \set{C}_i, \quad \forall i \in \{1, \dots, N\} \\
                & \btheta \in \bTheta.
\end{align*}
%
The second term of the objective does not occur in the parametric version of $\IOPd(\ell_{\mathrm{VI}}, \empiricaldist)$. This regularization term on model complexity is weighted by a parameter $\kappa > 0$ and can be interpreted as a penalty on the ``smoothness'' of the kernel representation \citep{girosi1993priors}. 

\begin{table}[!t]
    \resizebox{0.99\linewidth}{!}{%
    \begin{tabular}{p{0.28\linewidth}  p{0.72\linewidth}}
    \toprule
         \textbf{Risk measure} & \textbf{Summary} \\
         \midrule\textbf{Expected value}  & $\displaystyle\rho^\field{P}(\ell) := \field{E}_\field{P} \left[ \ell\left(\bx, \Xopt(\btheta) \right) \right]$   \\
            & \textbf{Best used for.~} Objectives and constraints for convex forward models. \\
            & \textbf{Statistical properties.~} Refer to Table~\ref{tab:summary_of_dd_losses}. \\
            & \textbf{Solution methods.~} Refer to Table~\ref{tab:summary_of_dd_losses}. \\
         \midrule\textbf{VaR}  & 
         $\displaystyle\rho^\field{P}(\ell, \chi) := \inf_\tau \left\{ \tau \;\Big|\; \field{P}\left\{ \ell(\bx, \Xopt(\btheta)) \leq \tau \right\} \geq \chi \right\}$   \\
           \citet{shahmoradi2021quantile} & \textbf{Best used for.~} Objectives for linear models. \\
            & \textbf{Statistical properties.~} Stability.\\
            & \textbf{Solution methods.~} Integer program reformulation. \\ 
         \midrule\textbf{CVaR} & 
         $\displaystyle\rho^\field{P}(\ell, \alpha) :=  \inf_\tau \tau + \frac{1}{\alpha} \field{E}_\field{P} \left[ \max\left\{ \ell(\bx, \Xopt(\btheta) - \tau , \; 0 \right\} \right]$ \\
            \citet{esfahani2018data}  & \textbf{Best used for.~} Objectives for linear and quadratic models. \\
            & \textbf{Statistical properties.~} Out-of-sample guarantee from distributionally robust optimization. \\
            & \textbf{Solution methods.~} Conic program reformulation for linear and some quadratic models. \\ 
         \bottomrule
    \end{tabular}
    }
    \caption{Different risk measures that can be minimized in data-driven inverse optimization. Each risk measure is accompanied by the formula, the forward model classes where it can be easily minimized, statistical properties that it satisfies, and the general solution method.}
    \label{tab:summary_of_dd_risks}
\end{table}

\begin{table}[!t]
    \centering
    \resizebox{0.95\linewidth}{!}{%
    \begin{tabular}{p{0.22\linewidth}  p{0.78\linewidth}}
    \toprule
         \textbf{Risk measure} & \textbf{Summary} \\
         \midrule\textbf{Distance} 
            & $\displaystyle\ell_\mathrm{D}\left((\bhx, \Xopt(\btheta)\right) := \min_{\bx \in \Xfeas(\btheta)} \norm{\bx - \bhx}_2$ \\
            & \textbf{Best used for.~} Objectives and constraints for convex models. \\
            & \textbf{Statistical properties.} Consistency~\citep{aswani2018inverse}.  \\
            & \textbf{Solution methods.~} Enumeration algorithm for general models; Kernel regression-based algorithm for strictly convex models~\citep{aswani2018inverse}; convex program when all models have the same feasible set~\citep{chan2018inverse, chan2018multiple}. \\
         \midrule\textbf{Absolute \qquad \quad Sub-optimality} 
            & $\displaystyle\ell_\mathrm{ASO}\left(\bhx, \Xopt(\btheta)\right) := \left| f(\bhx, \btheta) - \min_{\bx \in \Xfeas(\btheta)} f(\bx, \btheta) \right|$ \\
            & \textbf{Best used for.~} Objectives and constraints for linear models. \\
            & \textbf{Statistical properties.~} Generalization bound for linear models due to equivalence with VI. \\
            & \textbf{Solution methods.~} Linear program reformulation~\citep{chan2014generalized, chan2018inverse, chan2018multiple, esfahani2018data, ghobadi2021inferring}. \\
         \midrule\textbf{Relative \qquad \quad Sub-optimality} 
            & $\displaystyle\ell_\mathrm{RSO}\left(\bhx, \Xopt(\btheta)\right) := \left| \frac{f(\bhx, \btheta)}{\min_{\bx \in \Xfeas(\btheta)} f(\bx, \btheta)} - 1 \right|$ \\
            & \textbf{Best used for.~} Objectives for linear models. \\
            & \textbf{Statistical properties.~} N/A. \\
            & \textbf{Solution methods.~} Linear program reformulation~\citep{troutt2005linear, ref:troutt_ejor08, ref:troutt_ms06, chan2014generalized, chan2018inverse, chan2018multiple}. \\
         \midrule\textbf{Variational Inequality} 
            & $\displaystyle\ell_{\mathrm{VI}}\left( \bhx, \Xopt(\btheta) \right) := \max_{\bx \in \Xfeas(\btheta)} \;\;  \nabla_{\bx} f(\bhx, \btheta)^\tpose\left( \bhx - \bx \right)$ \\
            & \textbf{Best used for.~} Objectives for conic models. \\
            & \textbf{Statistical properties.~} Generalization bound~\citep{bertsimas2015data}. \\
            & \textbf{Solution methods.~} Convex reformulation~\citep{bertsimas2015data}. \\
         \midrule\textbf{KKT Conditions} 
            & $\displaystyle\ell_{\mathrm{KKT}}\left(\bhx, \Xopt(\btheta) \right) := \min_{\blambda \geq \bzero} \;\;  \ell_\mathrm{st}\left(\bhx, \Xopt(\btheta), \blambda\right) + \ell_\mathrm{pf}\left(\bhx, \Xopt(\btheta), \blambda \right)$ \\
            & \textbf{Best used for.~} Objectives for convex models. \\
            & \textbf{Statistical properties.~} N/A. \\
            & \textbf{Solution methods.~} Convex reformulation~\citep{keshavarz2011imputing}. \\
         \bottomrule
    \end{tabular}
    }
    \caption{Different loss functions that can be minimized in data-driven inverse optimization. Each loss function is accompanied by the formula, the forward model classes where the empirical risk measure can be easily minimized, statistical properties that it satisfies, and the general solution method.    }
    \label{tab:summary_of_dd_losses}
\end{table}

\subsection{Summary}

The data-driven inverse literature proposes risk measures and loss functions that are applicable for any convex forward optimization model. We summarize advantages of the different methods and discuss general rules-of-thumb to consider when selecting a model. The different methods contrast along two dimensions: (i) whether they provide useful statistical properties as an estimator of forward agent behavior and (ii) whether they admit efficient solution algorithms.

Table~\ref{tab:summary_of_dd_risks} summarizes common risk measures in the literature and highlights appropriate use cases.
While the empirical expected risk includes methods with statistical properties and fast solution algorithms, the quantile (VaR) and CVaR risk measures provide additional stability and generalization properties.
These may be useful when the inverse optimizer has limited data sets but must generate estimates that generalize to new instance or are robust to decision data sensitivity.

Table~\ref{tab:summary_of_dd_losses} summarizes the different loss functions under the empirical expected risk. 
All of these losses can be easily minimized for linear forward optimization models, with the Inverse Distance, KKT, and Variational Inequality losses being also efficient for conic models. 
Inverse Distance further guarantees a statistically consistent estimator, while Inverse Absolute Sub-optimality and Inverse Variational Inequality have generalization bounds; the former guarantees convergence to a ``true'' parameter given large data sets, while the latter bounds error when given finite data sets.

\section{Related Learning Paradigms}\label{sec:related_model_paradigms}

In this section, we overview two related streams to the inverse optimization literature. The first considers the situation where the decision data arrives sequentially in an online manner. This topic is relatively nascent but growing. The second is inverse reinforcement learning, which has evolved in parallel with inverse optimization into a divergent class of methods and applications.

\subsection{Inverse optimization through online learning}\label{subsec:online_learning}

Rather than using a fixed set of decisions to estimate a parameter in one shot, Inverse Online Learning (IOL) operates over multiple rounds where in each round $t \in \{1, 2, \dots, T\}$, the inverse optimizer observes a new decision $\bhx_t \in \set{X}_t$ and then constructs a parameter estimate $\btheta_{t}$ by updating an existing estimate $\btheta_{t-1}$.
Algorithm \ref{algo:Online_learning} outlines the general structure of an IOL problem.


\begin{algorithm}[t]
\caption{General inverse online learning framework}\label{algo:Online_learning}
\vspace{0.3cm}
\textbf{Input:} Data set $\{ \bhx_t, \Xfeas_t(\btheta) \}_{t = 1}^T$ arriving sequentially; Initial learning rate $\eta_1$; Initial estimate $\btheta_0$ \\
\textbf{Output:} Set of parameter estimates $\{\btheta_t\}_{t = 1}^T$ estimated sequentially
\begin{algorithmic}[1]
\For{$t = 1$ to $T$}
\State Update estimate $\btheta_{t}$ using $(\bhx_t, \Xfeas_t(\btheta_{t-1}), \eta_{t-1})$ 
\State Update learning rate $\eta_{t}$
\EndFor\\
\Return $\{\btheta_t\}_{t=1}^T$ 
\end{algorithmic}
\end{algorithm}

The core requirement of IOL algorithms is a computationally efficient update rule. Note that even if observing a sequential stream of decision data, we could still use traditional inverse optimization methods---which we refer to here as offline or ``batch'' mode---that use the entire data set to estimate a parameter. However, these approaches would require solving increasingly larger inverse problems from scratch in each round. In contrast, update rules improve the previous estimate by using only the most recent observed decision. In fact, the computational efficiency of the IOL algorithms make them efficient heuristics even for offline inverse optimization.

The computational efficiency of the update rules comes at a cost of the learning efficiency (i.e., estimation error) of the overall parameter. Thus, IOL methods require performance guarantees on this learning efficiency, represented by a regret function, which measures the difference in the cumulative loss from online learning versus the loss from a batch-level approach over the time horizon. For example, we may consider average regret 
\begin{align*}
    R\left(\{\btheta\}_{t=1}^T, \{\bhx_t\}_{t=1}^T \right) := \frac{1}{T} \Bigg( \, \, \underbrace{\sum_{t=1}^T \ell \left(\bhx_t, \Xopt_t(\btheta_t) \right)}_{\text{online performance}} \; - \ \underbrace{\min_{\tilde\btheta \in \bTheta} \sum_{t=1}^T \ell\left(\bhx_t, \Xopt_t(\tilde\btheta)\right)}_{\text{batch-level performance}} \, \, \Bigg),
\end{align*}
where $\ell(\bhx, \Xopt(\btheta))$ can be any of the loss functions considered in Section \ref{sec:data-driven_IO}. 
Then, the performance guarantee is derived as an upper bound of the rate of decrease of $R(\{\btheta\}_{t=1}^T, \{\bhx_t\}_{t=1}^T)$ as a function of the data sequence $\{\bhx_t\}_{t=1}^T$. This bound depends on the loss function and the update rule used. 

\subsubsection{Updates using a forward optimization oracle.}
\citet{barmann2020online} explore IOL by minimizing the Absolute Sub-optimality loss (see Section~\ref{subsec:suboptimality}). Consider the forward model
\begin{align}\label{eq:forward_model_online}
    \min_{\bx} \; \left\{\btheta^\top \bx \; \Big | \; \bx \in \mX_t \right\}
\end{align}
where the objective function is linear and unknown and the feasible set is an arbitrary bounded set. 
Furthermore, the authors assume access to a ``forward oracle'', such that whenever a new data point $(\bhx_{t}, \Xfeas_t)$ becomes available, an inverse optimizer can efficiently solve model \eqref{eq:forward_model_online} using the incumbent estimate $\btheta_{t-1}$ to generate an optimal solution $\bx^*_t \in \mX_t$. Then, the next estimate $\btheta_{t}$ is obtained with either of the two rules below, which are based on multiplicative weights update (MWU) \citep{arora2012multiplicative} and online gradient descent (OGD) \citep{zinkevich2003online} algorithms, respectively: 
%
\begin{align*}
     \prob{MWU}(\btheta_{t-1}, \bhx_t): \quad & \btheta_{t} = \btheta_{t-1} - \eta_{t} \, (\btheta_{t-1} \odot (\bx^*_t - \bhx_t))\\  
     \prob{OGD}(\btheta_{t-1}, \bhx_t): \quad & \btheta_{t} = \btheta_{t-1} - \eta_{t} \; (\bx^*_t - \bhx_t)
\end{align*}
Here, the symbol $\odot$ denotes element-wise vector multiplication. The specifics of the learning rate parameter $\eta_t$ can be found in \cite{barmann2020online}. The two update rules lead to slightly different performance guarantees, but both result in average regret converging at a rate of $\mO(1/\sqrt{T})$. However, the guarantees hold only under the assumption that the data is noise-free, i.e., all data points are generated exactly using a fixed $\btheta^*$ such that $\ell(\bhx_t, \Xopt_t(\btheta^*)) =  0$ for all $t$. This approach is extended by \citet{xinying2020online} with similar convergence results for certain classes of forward problems with non-linear objectives.

\subsubsection{Updates using an inverse optimization oracle.}\label{subsubsec:IOL_inverseoracle}

\citet{dong2018generalized} propose an update rule using Minimum Distance loss (see Section~\ref{sec:datadriven_distance}) for problems with convex bounded feasible regions and strictly convex objective functions. Their update rule requires the solution of a single-point inverse optimization problem to directly update their parameter estimates, i.e.,
\begin{align*}
    \btheta_{t} \in \argmin_{\btheta \in \Theta} \left\{ \norm{\btheta - \btheta_{t-1}}^2_2 + \eta_t \ell_\mathrm{D}(\bhx_t, \set{X}_t(\btheta)) \right\}.
\end{align*}
The authors also show that average regret converges at a rate of $\mO(1/\sqrt{T})$, but that this convergence rate holds even under noisy data, as long as some assumptions on the distribution of noise is met. 

\subsection{Inverse reinforcement learning}
\label{sec:irl_connections}

Inverse Reinforcement Learning (IRL) is the problem of estimating the reward function of a reinforcement learning problem~\citep{russell1998learning, ng2000algorithms}. 
Since reinforcement learning can be used for large-scale MDPs, the early IRL literature shares fundamental similarities to Inverse MDPs (see Section~\ref{subsec:MDP}). However, modern IRL methods draw from the machine learning literature (e.g., maximum likelihood estimation, gradient descent algorithms) more than inverse optimization~\citep{sutton2018reinforcement}. 
Nonetheless, the similarity of problems suggest that new data-driven inverse optimization methods may be obtained by adapting IRL techniques.
We briefly highlight the relationship between IRL and inverse MDPs, summarize early techniques, and sketch current directions; we refer to~\citet{arora2018survey} for a detailed survey.

A reinforcement learning problem is defined by the tuple $(\set{S}, \set{A}, p, \btheta, \gamma)$ comprising a state space, action set, transition probabilities, reward function, and discount factor. 
We deviate slightly from the notation of Section~\ref{subsec:MDP} and refer to specific rewards and value functions as $\theta(s, a)$ and $v(s)$, respectively.
We assume without loss of generality that there is an initial state $s_0$ from which all trajectories begin.
In IRL, we may observe a policy $\hat\bpi(s)$ from an agent, or instead only observe a data set of length $T$ state-action trajectories $\dataset = \{ \tau_i \}_{i=1}^N$ where $\tau_i := \langle (s_0, a^i_0), (s^i_1, a^i_1), \dots, (s^i_T, a^i_T) \rangle$. 
The goal is to estimate a reward vector $\btheta$ for which the observed policy or trajectories are optimal.

Early IRL methods observe a policy $\bhpi$ and solve an inverse problem with convex programming. 
For instance, Max-Margin methods estimate a reward function such that the actions taken by $\bhpi$ achieve higher expected rewards than any other actions~\citep{russell1998learning, ng2000algorithms, abbeel2004apprenticeship, ratliff2006maximum}.
Let
\begin{align*}
    q(s, a) := \theta(s, a) + \gamma \sum_{s' \in \set{S}} p(s' | s, a) v(s')
\end{align*}
denote the state-action $q$-function. Then, given a policy $\bhpi$, the Max-Margin loss function is
\begin{align*}
     \ell_\mathrm{MM}(\bhpi) := \sum_{s \in \set{S}} \left( q(s, \bhpi(s)) - \max_{a \in \set{A} \setminus \{\bhpi(s)\}} q(s, a) \right).
\end{align*}
The above loss computes the difference in $q$-values between the actions from an observed policy and the next best action. 
For MDPs where the reward only depends on the state, i.e., $\theta(s, a) = \theta(s)$ for all $a \in \set{A}$,~\citet{ng2000algorithms} note that this loss is a convex function of $\btheta$. They formulate a convex program where the estimated reward is constrained to ensure that the observed policy satisfies an optimality condition for the MDP. Below, we present a variant of the original formulation:
\begin{align}\label{eq:irl_max_margin}
    \begin{split}
        \max_{\btheta, v, q} \quad     & \ell_\mathrm{MM}(\bhpi) - \kappa \norm{\btheta}_1 \\
        \st \quad   & q(s, a) = \theta(s) + \gamma \sum_{s' \in \set{S}} p(s' | s, a) v(s'), \quad \forall s \in \set{S}, \forall a \in \set{A} \\
                    & v(s) \geq q(s, a), \quad \forall s \in \set{S}, \; a \in \set{A} \\
                    & v(s) = q(s, a'), \quad \forall s \in \set{S}, \; a' = \bhpi(s) \\
                    & \btheta \in \bTheta.  
    \end{split}
\end{align}
The constraints of Problem~\eqref{eq:irl_max_margin} are equivalent to~\eqref{eq:inverse_mdp2}--\eqref{eq:inverse_mdp4} except here written with the $q$-function.
Furthermore in addition to the Max-Margin loss,~\citet{ng2000algorithms} include a regularization term $-\kappa \norm{\btheta}_1$ that encourages ``simpler'' small-magnitude estimated rewards. 

Problem~\eqref{eq:irl_max_margin} can be typically too difficult to solve for practical applications where the state and action spaces are large and we only observe trajectories rather than policies and state transition probabilities. 
The Max-Margin literature proposes two resolutions to address these concerns. 
First, we may model the reward as a linear combination of fixed basis functions $\theta(s) = \sum_{b=1}^B \theta_b f^{(b)}(s)$; 
this can reduce the number of variables to estimate and further ensure linearity in the parameters identical to convex-separable bases in the inverse optimization literature.
Second, we may eschew designing a reward function itself and instead minimize the margin between the value function of the optimal policy $\bpi(\btheta)$ and the empirical expected state-action $q$-function value, i.e., 
\begin{align*}
    \min_{\btheta \in \bTheta} \quad & \ell \left( v^{\bpi(\btheta)}\left(s_0 \right) , \; \frac{1}{N} \sum_{i=1}^N \left( \theta\left(s_0\right) + \sum_{t=1}^T \gamma^t \theta\left(s_t^{(i)}\right) \right) \right)
\end{align*}
where $\{\langle (s_0, a^i_0), (s_1^i, a_1^i), \dots, s_T^i, a_T^i) \rangle \}_{i=1}^N$ is the data set of observed trajectories and $\ell(v_1, v_2)$ is a penalty function on the difference of values. The above problem is typically solved using iterative heuristics that compute the margin with respect to a set of candidate policies $\tilde\bpi$ in place of the optimal policy. Algorithm~\ref{algo:inverse_reinforcement_learning} highlights the general steps of a Max-Margin method.

\begin{algorithm}[t]
\caption{General max-margin inverse reinforcement learning framework}
\label{algo:inverse_reinforcement_learning}
\vspace{0.3cm}
\textbf{Input:} Data set of trajectories $\{\tau_i\}_{i=1}^N$\\
\textbf{Output:} Reward function estimate $\hat\theta(s, a) = \sum_{b=1}^B \hat\theta_b f^{(b)}(s, a)$
\begin{algorithmic}[1]
\State Initialize a reward estimate $\tilde\btheta$; Memory of rewards $\set{P} = \emptyset$
\While{Convergence criteria not met}
    \State Solve the reinforcement learning problem with $\tilde\btheta$ to obtain a candidate policy
    \begin{align*}
        \tilde\bpi \gets \text{RL}(\set{S}, \set{A}, p, \tilde\btheta, \gamma) 
    \end{align*}
    \State Update $\set{P} \gets \set{P} \cup \{ \tilde\bpi \}$ and solve an approximate Max-Margin IRL problem, e.g.,
    \begin{align*}
        \tilde\btheta \gets \argmin_{\btheta \in \bTheta} \quad & \sum_{\tilde\bpi \in \set{P}}  \ell \left( v^{\tilde\bpi}\left(s_0 \right) , \; \frac{1}{N} \sum_{i=1}^N \left( \theta\left(s_0\right) + \sum_{t=1}^T \gamma^t \theta\left(s_t^{(i)}\right) \right) \right)
    \end{align*}
\EndWhile \\
\Return Final reward estimate $\tilde\btheta$
\end{algorithmic}
\end{algorithm}

Note that the problem formulation of large-scale Max-Margin methods parallel the formulation of data-driven inverse optimization. In inverse optimization, inverse-feasibility becomes hard to satisfy with large decision data sets, necessitating data-driven loss functions that penalize the violation of optimality conditions. On the other hand, classical IRL is too large to model using optimality criteria, leading to loss functions that penalize the sub-optimality of the trajectories.

Recent IRL methods relate more closely with modern machine learning, e.g., Maximum Entropy~\citep{ziebart2008maximum, boularias2011relative}, Bayesian IRL~\citep{ramachandran2007bayesian, lopes2009active, levine2011nonlinear}, or supervised learning of $q$-values~\citep{taskar2005learning}. For example, the Maximum Entropy literature posits that in an MDP with stochastic state dynamics, the observed trajectories $\tau$ are drawn from an optimal policy and must have the highest likelihood over all other trajectories. 
This can recast IRL as a maximum likelihood problem $\max_{\theta} \; \field{E}_{\field{P}} \left[ \log p_{\theta}(\tau) \right]$ where $p_{\theta}(\tau) \propto p(s_0) \prod_{t=0}^T p(s_{t+1}|s_{t}, a_{t}) \exp(\gamma^{t} \theta(s_t, a_t))$ follows a Boltzmann distribution.

We conclude this section by highlighting a common weakness of both inverse optimization and IRL. Both problems face the risk of obtaining uninformative degenerate estimates. Note that for any MDP, the all-zero reward function $\theta(s, a) = 0$ will ensure that every policy (including an observed one) is optimal. In IRL, this necessitates well-designed objectives such as the Max-Margin. However,~\citet{ng1999policyinvariance} further show that the optimal policy for an MDP with reward function $\theta(s, a, s')$ is invariant to the transformation $\theta(s, a, s') + \gamma \Phi(s') - \Phi(s)$ for any \emph{arbitrary} function $\Phi: \set{S} \rightarrow \field{R}$. Effectively, if a policy is optimal, there are an infinite number of reward functions for which it is so. 
To mitigate this problem,~\citet{fu2017learning} propose an adversarial IRL framework that estimates `disentangled' rewards. 
Although not to the same extent, the scaling invariance property (see Section~\ref{subsubsec:data-driven_IO_rdg}) is a similar characteristic in inverse optimization. However, scaling invariance is often easily mitigated by a normalization constraint in inverse optimization.

\section{Applications in Design Problems} \label{sec:classic_IO_apps} 

We review the application of classical inverse optimization for the design of price incentives in bilevel games (Section~\ref{subsec:incentive_design}) and price mechanisms in cooperative games (Section~\ref{subsec:mechanism_design}).
These two problems encompass many application areas and share a common theme of using prices to \emph{induce} behavioral change, under the assumption that the original decision-generating processes are known.

\subsection{Incentive design and pricing in bilevel games}\label{subsec:incentive_design}

Bilevel games represent hierarchical decision-making problems in which a leader makes the first decision and then a follower responds with their own decision \citep{colson2007overview}. The leader may represent a policy-maker that generates incentives $\btheta \in \bTheta$ to encourage specific follower decisions that, for example, may be environmentally friendly or socially responsible. When the incentives are monetary, they appear as perturbations to the objective of the follower's decision-making process
\begin{align}
\min_{\bx} \; \left \{ g(\bx, \btheta) \; | \; \bx \in \mX \right \}. \label{eq:follower}
\end{align}
Assuming that the follower's problem is fully observable to the leader, the leader's problem is
\begin{align}
\min_{\bx, \btheta} \; \left \{  f(\bx)  \; \bigg | \; \bx \in  \argmin_{\bx \in \mX} g(\bx, \btheta), \, \btheta \in \bTheta  \right \},  \label{eq:leader} 
\end{align}
where $f(\bx)$ may, for example, represent the environmental cost of follower decision $\bx$.

While the leader's problem can sometimes be reformulated as a single, large monolothic formulation, a more efficient approach is to decouple the computation of follower ``target" decisions and incentives, i.e., as a master problem and an inverse optimization subproblem. This process is depicted in Figure \ref{fig:app_classical_IO}. In the master problem, the leader solves a variant of problem~\eqref{eq:leader} where they minimize their own objective assuming control of the follower's decision. This generates a ``target" decision $\bhx$. To compute incentives $\btheta \in \bTheta$ that render $\bhx$ an optimal solution to the follower's problem, the inverse optimization subproblem is solved for forward problem~\eqref{eq:follower}. We describe several examples of applications where this decomposition approach is used. 

\begin{figure}[t]
    \centering
    \scalebox{0.9}{%
\begin{tikzpicture}[font=\sffamily]
    \draw (0, 0)    node[box, minimum height=1.7cm, minimum width=3cm, text width=3.0cm] (mp) {Master \\ problem};
    \draw (4, 0)    node[text width=2cm, align=center] (targets) {Target decision(s)};
    \draw (4, -3)   node[minimum height=1.7cm, minimum width=3cm, text width=3.0cm, align=center] (known) {Follower \\ decision-making \\ model(s)};
    \draw (8, -1.5)   node[box, minimum height=1.7cm, minimum width=3cm, text width=3.0cm] (io) {Inverse \\ optimization \\ sub-problem};
    
    \draw (13, -1.5)   node[text width=3.5cm, align=center] (incentives) {Incentives and price mechanisms};
    
    \draw[arrow] (mp) -- (targets);
    \draw[arrow] (targets) -- (io);
    \draw[arrow] (known) -- (io);
    \draw[arrow] (io) -- (incentives);
\end{tikzpicture}
}





    \caption{The solution process of many bilevel optimization problems. In this pipeline, classical inverse optimization is used to compute incentives that can induce followers to generate decisions that are optimal for the leader. These decisions are referred to as target decisions.}
    \label{fig:app_classical_IO}
\end{figure}
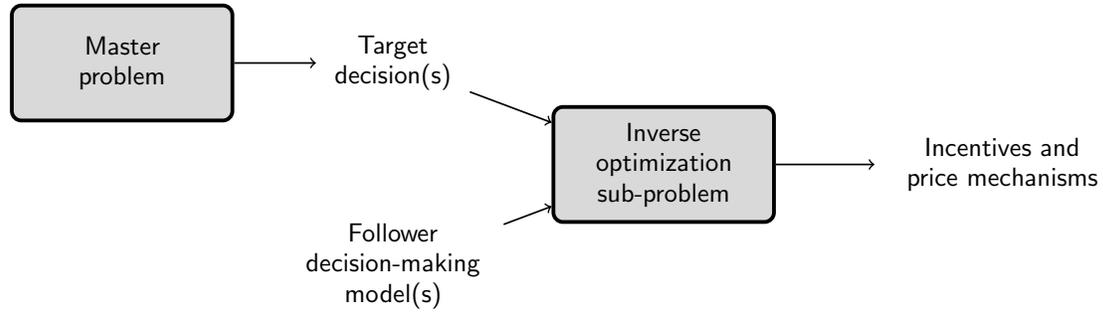

\cite{marcotte2009toll} study the toll design problem for diverting the transportation of hazardous materials away from high-risk and high-population areas. The leader is a local government imposing tolls while the follower is the carrier that is solving a linear optimization routing problem. The authors show that while the leader's problem can be solved as a large mixed integer program, the master-subproblem decomposition is more efficient and has the added benefit of computing tolls that optimize for another criteria in addition to making a low-risk route optimal. Details of this formulation can found in Section \ref{ECsubsec:toll_design}. \cite{esfandeh2016regulating} and \cite{bianco2016game} consider extensions of this toll design problem by accounting for network congestion effects and interactions between competing carriers, respectively. The Variational Inequality and KKT optimality conditions are leveraged to solve the corresponding inverse problems.


Similar decomposition approaches have been examined for bilevel games aimed at mitigating carbon emissions \citep{rathore2021differential}.
For example, \citet{zhou2011designing} consider a setting where the follower is an electricity generator that is planning to expand the generation capacity of different resources. A Partial Inverse Optimization problem (see Section \ref{subsubsec:Partial_IO}) is used to compute carbon tax and subsidy schemes to incentivize greater investment into renewable generation methods. Finally, the classical inverse optimization problem is also found in bilevel revenue maximization problems. Here, the leader aims to maximize total revenue collected from followers. \cite{brotcorne2008joint}, \cite{brotcorne2011exact}, and \cite{afcsar2021revenue} show that leader's problem can be solved by decomposing it into a series of inverse optimization problems where each problem computes the maximum amount of profit that can be achieved for a given follower decision.

\subsection{Mechanism design in cooperative games}\label{subsec:mechanism_design}


Consider a transportation market with $N$ carriers, each facing demand to transport people or goods across a network $\mG(\mV, \mA)$ with nodes $\mV$ and directed arcs $\mA$. Each carrier $i \in \{1, \ldots, N\}$ uses a set of capacitated transportation resources to obtain profits $r_{(v_1,v_2),i}$ for each unit of demand satisfied between node pairs $v_1,v_2 \in \mV$, up to a certain limit. Examples of carriers include freight transportation companies \citep{agarwal2010network} or airlines \citep{houghtalen2011designing}. 
While carriers can operate independently, the demand for each carrier may not match their resource capacity constraints. Thus, the total profits gained from behaving cooperatively, i.e., as an alliance, is greater than if carriers behave independently.

Inverse optimization can be used to compute profit-sharing mechanisms that incentivize carriers to make decisions that collectively lead to a cooperatively optimal solution \citep{agarwal2008mechanism,agarwal2010network, houghtalen2011designing, zheng2015network, zheng2015empty}. For clarity, we consider a simplified model from this literature by assuming that $\mG$ is a fully connected network, so that we may simplify the reference of each pair of nodes $v_1, v_2 \in \mV $ to arc notation $a \in \mA$. We also assume zero operating costs for each carrier. Suppose that there is a centralized operator who can assume control over all resources from each carrier $i$. This operator can generate any decision $\bx \in \Xfeas$ where each variable $x_a$ denotes the flow on arc $a$ and $\Xfeas$ denotes flows that do not exceed collective carrier demands and resource constraints. A cooperatively optimal solution $\bhx \in \mX$ can be generated by solving the following master problem, which maximizes the collective profit obtained over the entire network: 
\begin{align*}
\max_{\bx}  \; \left \{ \sum_{i = 1}^N \, \sum_{a \in \mA} r_{a, i} x_{a}\; \bigg | \; \bx \in \mX \; \right \}.
\end{align*}
A subsequent assignment problem decomposes $\bhx$ into flows for each carrier $\bhx_i$ such that $\bhx = \sum_{i = 1}^N \bhx_i$, where each $\bhx_i \in \mX_i$ represents a ``target'' decision for carrier $i$.

In an alliance, each carrier has full control over their resources and can make any decision $\bx_i \in \mX_i$. One way to incentivize the target decisions is through a 
cost-per-arc $\theta_a$ agreement. This mechanism rewards carriers for sharing resource capacity along arcs $a \in \mA$, and penalizes them for using their resources to satisfy individual own demand. Carrier $i$'s problem under this mechanism is
\begin{align*}
\max_{\bx_i^i, \bx^{-i}_i} \; \left \{ \sum_{a \in \mA} r_{a,i} x^i_{a,i} +  \sum_{a \in \mA} \Big (\gamma_{a,i} \: x^{-i}_{a, i} - (1-\gamma_{a,i}) x^i_{a, i} \Big ) \theta_{a} \; \bigg | \; \bx_i \in \mX_i \; \right \},
\end{align*}
where the parameter $\gamma_{a,i}$ denotes a pre-defined fraction of the arc cost $\theta_a$ that must be paid by carrier $i$ to use arc $a$, and similarly, the fraction of arc cost $\theta_a$ that is received if $i$ shares its resources along arc $a$ \citep{agarwal2010network}. 
The variables $x^i_{a,i}$ and $x^{-i}_{a,i}$, where $x_{a,i} = x^i_{a,i} + x^{-i}_{a,i}$, denote the subset of resources that are used to satisfy $i$'s demand and the demand of all other carriers, respectively. To compute a single vector $\btheta = (\theta_1, \ldots, \theta_{|\mA|})$ for which every $\bhx_i$ is inverse-feasible, we solve the inverse optimization problem
\begin{equation}\label{eq:collaborative_IO}
    \min_{\btheta}  \; \left \{ 0 \; \Big| \; \bhx_i \in \Xopt_i(\btheta) \ \forall i \in \{1,\ldots,N\}, \; \btheta \in \bTheta \; \right \}.
\end{equation}
An important feature of this inverse model is that $\btheta$ must satisfy a large number of other constraints, denoted by set $\bTheta$, which include budget-balancing and stability constraints. The latter ensures that all carrier decisions are in the core of the cooperative game \citep{lucas1971some}, i.e., that no subset of carriers finds it more profitable to form a subcoalition \citep{agarwal2010network}. 
This inverse problem is solved by reformulating the carriers' forward problems using complementary slackness or KKT conditions \citep{agarwal2008mechanism,agarwal2010network, houghtalen2011designing, zheng2015network, zheng2015empty}.


\section{Applications in Estimation Problems}
\label{sec:modern_applications}

Next, we discuss applications in transportation (Section \ref{subsec:transportation}), power systems (Section \ref{subsec:markets}), and healthcare (Section \ref{subsec:healthcare}). 
In Section \ref{subsec:other_apps}, we briefly highlight additional application areas.

\subsection{Transportation and logistics}\label{subsec:transportation}



\subsubsection{Transportation system modeling.}\label{subsubsec:system_trans_IO} 

Equilibrium models have a rich history of modeling congestion in road networks \citep{patriksson2015traffic}. In these models, the travel time on each link in the network is described using a cost function that depends on volume of traffic flow on the link, and overall traffic in the network is assumed to satisfy a set of equilibrium conditions defined by these cost functions. However, finding the ``right" cost functions to describe a specific network is challenging. \cite{bertsimas2015data} apply the non-parametric Inverse Variational Inequality model (see Section \ref{subsec:non-param}) to estimate these cost functions using data of historical traffic flows in a given geographical region. Under the computed cost functions, the observed traffic flow data approximately satisfy the equilibrium conditions, which are in fact the optimality conditions of a convex forward problem. We leave the mathematical details to Section~\ref{subEC:traffic_VI}. 

The models that are estimated using the inverse optimization approach have been shown to produce high-quality forecasts of traffic flow across a number of different real-world benchmarks \citep{bertsimas2015data, zhang2016price, zhang2018price}. Aside from prediction, the estimated cost functions provide insight into network inefficiency and the value of smart traffic control \citep{zhang2018price}. Furthermore, they can also be used to guide infrastructural investment and policy decisions, for example, by estimating the effects of new bike lanes on existing road networks \citep{liu2022planning}. In this context, inverse formulations can be developed for richer equilibrium models. For example, \citet{zhang2017data} consider inverse problems for multi-class transportation models where different vehicle types contribute differently to congestion levels.

\subsubsection{Personalized agent models.}\label{subsubsec:agent_trans_IO}

Prescriptive models for routing and scheduling in the transportation literature typically assume that the objective function describing the ``costs" of every decision is given. In practice, these costs may be unknown, vaguely defined or difficult to characterize. For instance, even under the exact same conditions, two individuals may prefer different decisions due to nonidentical, subjective preferences over various factors. The ``costs" in these models thus require estimation.

\textbf{Route recommendation.} 
\citet{ronnqvist2017calibrated} describe a widely-used route recommendation system for long haul truck drivers in the Swedish forest industry. The recommendation system (i.e., the forward problem) is a multi-objective minimum cost routing problem where the feasible set $\mX = \{\bx \; | \; \bA\bx = \bb\}$ is defined by node-arc incidence matrix $\bA$ and vector $\bb$ capturing the origin and destination node. The ``cost" for any feasible path $\bx \in \mX$ is a weighted sum of $B$ different objective functions $f(\bx) = \sum_{b=1}^B \theta_b f^{(b)}(\bx)$, where each objective $f^{(b)}(\bx) = \bu_b^\top \bx$ represents a linear penalty such as fuel costs, road quality, speed limits and driver safety. To calibrate weights associated with each objective function, a sample of ``preferred routes'' between various origin and destination nodes are collected from the drivers. Given these routes $\bhx_i$ for $i \in \{1,\ldots, N\}$, an inverse optimization model that minimizes the Absolute Sub-optimality loss (Section \ref{subsubsec:abs_suboptimality}) is solved: 
    \begin{align*}
    \begin{split}
        \min_{\btheta, \blambda_i} \quad & \frac{1}{N} \sum_{i=1}^N w_i  \bigg | \, \Big( \sum_{b=1}^B \theta_b \bu_b^\tpose \Big)\, \bhx_i - \bb_i^\tpose \blambda_i \, \bigg | \\
        \st \quad   & \bA^\tpose \blambda_i = \sum_{b = 1}^B \theta_b \bu_b, \quad \blambda_i \geq \bzero, \quad \forall i \in \{1, \dots, N\} \\
                    & \btheta \in \bTheta.
    \end{split}
    \end{align*}
    The weights $w_i$ associated with each data point $i$ prioritize different preferred paths, for example, by giving a higher value to paths between nodes that are farther apart.

\citet{chen2021inverse} consider inverse optimization for capacitated vehicle routing models, where routes chosen by long-term ``expert" drivers are used to refine travel time estimates in a network. Given a network with $m$ nodes and $n$ arcs, the routing problem seeks to find a minimum cost route that satisfies all demand for a homogeneous good at every node using a capacitated truck that replenishes inventory by revisiting depots. Let $Q$ denote the capacity of vehicle. Let $\bx$ represent arc choices and $\by \in \mathbb{R}^m_+$ represent the amount of a good in the vehicle at each node on the route. Then, given a demand vector $\bpi \in \mathbb{R}^m_+$, the forward routing problem can be formulated as the following mixed integer linear optimization model
    \begin{align*}
    \min_{\bx, \by} \left \{\btheta^\top \bx \; \Big | \; \bA\bx + \bB\by \geq \bd, \; \bpi \leq \by \leq Q\mathbf{1}, \; \bx \in \{0,1\}^n \right \}.
    \end{align*}
    Given a dataset $\{(\bhx_i, \bhy_i), \bpi_i\}_{i=1}^N$,~\citet{chen2021inverse} estimate $\btheta$ using Inverse Online Learning with the Multiplicative Weights Update rule (see Section \ref{subsec:online_learning}). 
    The authors show in a case study with an online retailer that the calibrated model can reproduce many ``expert" decisions.

\textbf{Agent modeling and policy-making.}
Third party decision-makers such as urban planners or public agencies can also use calibrate agent models to help guide policy decisions. \cite{ref:chow_12}, \cite{kang2013location} and \cite{chow2015activity} use inverse optimization to calibrate ``household activity" models, which are mixed integer programming problems that model individual spatiotemporal travel patterns. The authors use real household survey data from cities across North America to calibrate these models, which can then be used to evaluate the effects of new investments in transit. 
Alternatively, \citet{you2016inverse} use GPS data to calibrate vehicle routing models describing the freight activity of trucking companies. These models help urban planners better predict potential changes in trucking activity when introducing new parking, delivery or distribution regulations. \citet{wei2018modeling} calibrate airline crew scheduling models using real data of crew schedules from a regional carrier in the U.S.. The calibrated models provide insights into how different airlines handle scheduling delays and how air traffic controllers can better manage centralized operations. With the exception of \citet{wei2018modeling}, where the authors develop their own local search heuristic to solve the inverse mixed integer optimization model, the rest of this agent modeling literature employs an algorithm 
that closely resembles Inverse Online Learning with an inverse optimization oracle (see Section \ref{subsec:online_learning}). The cutting plane algorithm, discussed in Section \ref{subsec:integer}, is used to solve the oracle.


\subsection{Power systems}\label{subsec:markets}



\subsubsection{Market-clearing models.}\label{subsubsec:market-clearing}
In North America, electricity markets are typically supported by a market operator, whose role includes scheduling production and power flow by matching supply offers with demand bids for future dates subject to certain institutional constraints. This market-clearing process is generally formulated as a linear optimization model and is solved at regular intervals every day. Furthermore, the final allocations (solutions), prices (shadow prices), and 
bid functions are published to maintain transparency in the market-clearing process. \cite{birge2017inverse} describe an inverse optimization model that uses this public information to infer unobservable model parameters reflecting institutional constraints. Learning about these constraints can provide market participants and policy-makers with a better understanding of transmission capacities, which can inform bidding strategies or infrastructural investment decisions. The inverse model assumes a noise-free environment in which the constraint matrix to be estimated must make the data feasible and satisfy complementary slackness. Details of this estimation procedure are provided in Section \ref{subEC:market_clearing}. 

This application also demonstrates how optimal dual variables of the forward problem can be both meaningful and directly accessible as data. These values can help restrict the parameter search space and may allow for more precise parameter estimation using fewer data points.

\subsubsection{Agent demand models.} \label{subsubsec:consumption_models}
%
Price-responsive consumers of electricity include smart buildings, electric vehicles and data centers. Accurate demand models of these consumers can help operators, service providers and retailers make better planning, pricing and bidding decisions. We describe the use of inverse optimization to estimate utility functions of consumer demand models, which can then be used to generate demand predictions.

\paragraph{Forward model.} 

Let $y$ be a variable representing a level of electricity consumption. Given a price $p_t$ at time period $t$, we assume that an agent's electricity consumption behaviour is the solution to an unobserved constrained utility maximization problem
\begin{equation}\label{model:elec_cons_init}
    \max_{y} \; \left \{ u_t(y) - p_ty \; | \; l_t \leq y \leq r_t \right \},
\end{equation}
where $u_t(y)$ is the agent's utility function and $(l_t, r_t)$ represent known lower and upper bounds on consumption defined by physical properties such as building thermal dynamics \citep{fernandez2021forecasting}, or operational constraints such as building temperature settings \citep{saez2016data} 
and electric vehicle charging levels \citep{fernandez2019ev}. We assume that the utility function is a stepwise function
\begin{equation*}
  u_t(y) :=
  \begin{cases}
  c_{1,t}y \quad & \text{if $y < s$,}\\
  c_{1,t}s_{1} + \ldots + c_{i,t}(y - (i-1)s)  \quad & \text{if $(i-1)s \leq y \leq is, \ i \geq 2$,}
  \end{cases}
\end{equation*}
where each ``step" has equal width $s$ for a total of $I$ steps. Furthermore, we assume that $c_{i}^t \geq c^t_{j}$ for all $i < j$, which implies that the marginal utility is non-increasing in the consumption amount. With this function, model~\eqref{model:elec_cons_init} can be rewritten as a linear program 
\begin{align*}
\max_{\bx} \;  \left \{ \left(\, \sum_{i = 1}^I c_{i,t} x_i \,\right) - p_t\mathbf{1}^\top \bx \; \Bigg | \; l_t \leq \mathbf{1}^\top \bx \leq r_t \right \} 
\end{align*}
with decision vector $\bx$ representing the amount consumed in each step.

\paragraph{Inverse model.}
The inverse optimization model estimates the components $c_{i,t}$ of the utility function for all $i$.
Let the vector $\bpi_t = (\pi_{t,1}, \ldots, \pi_{t,m})$ describe a set of $m$ ``features'' at period $t$. Features can include recent consumption behavior or descriptions of weather conditions. 
Consider a data set of $N$ observations of consumption-features-price tuples $\{(\hat y_t, p_t, \bpi_t)\}_{t=1}^N$.
We first transform each $\hat y_t$ into a vector $\bhx_t$ by setting the first $\lfloor y_t/s \rfloor$ elements of the vector as $s$, the $\lceil y_t/s \rceil$-th element as $\lfloor y_t/s \rfloor -s$, and the last $I-\lceil y_t/s \rceil$ elements as zero. 
We then define the utility function components as a linear regression of the features
\begin{align*}
    c_{i,t} = \theta_{i,0} + \theta_{i,1}\pi_{t,1} + \ldots + \theta_{i,m}\pi_{t, m}, \quad \forall i \in \{1, \dots, I\}.
\end{align*}
Let $\kappa >0$ be a regularization parameter. The regularized inverse optimization problem is 
\begin{align*}
\min_{\btheta \in \bTheta} \; \frac{1}{N} \sum_{t = 1}^N \ell \left (\bhx_t, \Xopt_t(\btheta) \right ) + \kappa \norm{\btheta}_p,
\end{align*}
which estimates the regression parameters of the utility function subject to the utility taking the shape of a stepwise non-increasing linear function.
Both the Minimum Distance loss function $\ell_{\mathrm{D}}(\bhx_t, \Xopt_t(\btheta))$ (Section \ref{sec:datadriven_distance}) and the Absolute Sub-optimality loss function $\ell_{\mathrm{ASO}}(\bhx_t, \Xopt_t(\btheta))$ (Section \ref{subsec:suboptimality}) have been examined in the literature \citep{saez2016data,saez2017short,fernandez2019ev,fernandez2021forecasting,lu2018data}. 
The regularization hyperparameter $\kappa$ can be tuned using cross-validation.  Since model \eqref{model:elec_cons_init} is a linear optimization model, strong duality can be used to formulate and solve the inverse problem.

The calibrated forward models are shown to outperform both classical and black-box time series forecasting methods over synthetic and real data. 
The estimated utility function also possesses the same stepwise linear shape as the demand bids that must be submitted to market-clearing operators, meaning that the estimated functions can be directly submitted as a bid by a retailer. This latter observation reveals another benefit of using such a ``structured" estimation framework.


\subsection{Healthcare}\label{subsec:healthcare}


\subsubsection{Healthcare systems.}



\paragraph{Clinical pathway concordance measurement.}

Clinical pathways specify standardized processes in healthcare delivery for a specific group of patients. Patient journeys through the healthcare system that follow these pathways are considered ``concordant'', whereas alternate pathways that ignore certain steps or include extraneous steps are deemed ``discordant''. \citet{chan2022inverse} use inverse optimization to develop a quantitative metric that measures the concordance of patient-traverse pathways against the recommended clinical pathways. Patient pathways are modeled as a walk on a graph representing the healthcare system. Clinical pathways are shortest paths on this graph. The following model represents an inverse shortest path problem, where the goal is to find arc costs $\btheta$ that make the clinical pathways shortest paths.
\begin{align} \label{model:IOclinpathway}
\begin{split}
\min_{\btheta, \blambda,\boldsymbol{\epsilon}} \quad & \sum_{i=1}^{N}(\epsilon^{r}_{i})^2\\
\st \quad & \bA^\tpose \blambda \le \btheta\\
&  \btheta^\tpose \hat\bx^{r}_{i}=\bb^\tpose \blambda+\epsilon^{r}_{i}, \quad  \forall i \in \{1, \ldots, N \} \\
&  \|\btheta\|_{\infty}= 1 \\
& \bA \btheta= \bzero.
\end{split}
\end{align}
In this model, $\bA$ and $\bb$ encode the flow balance constraints in a shortest path problem with one source and one sink. We assume there are $N$ clinical pathways given, $\hat\bx^r_1, \ldots, \hat\bx^r_N$, between the source and sink. This model is a slight variation of the Absolute Sub-optimality model~\eqref{eq:aso_lp_form_one_feas_region}, where: (i) the objective is the sum of squared duality gaps instead of sum of absolute duality gaps, and (ii) the duality constraints are modified, since the shortest path forward problem is a standard form linear program. The normalization constraint $\|\btheta\|_\infty = 1$ prevents $\btheta = \bzero$ from being optimal. The constraint $\bA\btheta = \bzero$ deals with the lower dimensionality of the forward feasible region (due to flow balance equality constraints) and ensures the cost vector does not point orthogonal to the entire feasible region. These constraints form $\bTheta$. Additional details are given in Section~\ref{sec:DPMEC}.

This work has been extended to the more complex clinical pathways associated with breast cancer diagnosis and treatment, modeled using hierarchical networks~\citep{chan2021inverse_hier}.



\paragraph{Healthcare contract design.}

The United States Medicare Shared Savings Program (MSSP) offers providers of medical care, known as Accountable Care Organizations (ACOs), bonus payments if they can reduce the cost of providing care for Medicare patients, as long as certain quality constraints are maintained. 
%
\citet{aswani2019data} propose a modified MSSP contract with a performance-based subsidy for an ACO's upfront investment, to make participation in the MSSP more attractive. They use a principal-agent framework, where Medicare is the principal (aiming to maximize savings) and an ACO is the agent (aiming to maximize expected payoff).
%
The resulting problem is bilevel since the solution to the ACO's problem, its optimal savings, is embedded in the constraints of Medicare's contract design problem. The key parameter to be estimated is an ACO's ``type'', which is unobservable to Medicare and represents the ACO's ability to generate savings. 
Using a public data set 
that includes data on ACO spending and savings,~\citet{aswani2019data} solves a data-driven inverse optimization model to estimate the distribution of ACO types from the data. They use this distribution to construct optimal subsidy-based contracts that can simultaneously increase Medicare savings and ACO payments.


\paragraph{Breast cancer screening.}
Disease screening strategies depend on the sensitivity and specificity of the test.~\citet{ayer2015inverse} estimates the range of sensitivity and specificity values of a hypothetical test that would make a given screening policy optimal. Disease progression under a given screening policy is modeled as a partially observable Markov chain. The resulting inverse optimization problem becomes a nonlinear problem and a complete solution method is proposed. 
%

\subsubsection{Personalized decision making.}


\paragraph{Radiation therapy treatment planning.}

Radiation therapy (RT) is one of the main ways to treat cancer. 
Treatment design --  the angles from which the radiation is delivered, the shape of the radiation beam from each angle, and the distribution of intensity over the beam's shape (divided into small beamlets) -- is performed by solving a large optimization model with an objective that is a weighted sum of multiple sub-objectives. Inverse optimization learns these weights from historical treatments. Using real data from previously treated prostate cancer patients,~\citet{chan2014generalized} show that simple linear models with inversely optimized objective function weights for a small number of objectives could recreate treatments that were designed using complex, non-convex treatment planning models with many objectives. Section~\ref{sec:RTEC} provides more details.

Extensions from a modeling perspective include using convex objective functions \citep{chan2018trade, sayre2014automatic}, learning objective functions rather than the weights~\citep{ajayi2022objective}, and using input data that represents partial dose distributions  \citep{babier2018knowledge}. Another direction of research includes combining inverse optimization with machine learning methods to generate personalized weights for each patient based on patient-specific anatomy, and to automate the treatment planning process for prospective patients~\citep{lee2013predicting,boutilier2015models,babier2018knowledge,babier2020knowledge,babier2020importance,goli2018small}. Finally,~\citet{ghate2020imputing} uses inverse optimization to estimate unobservable parameters of a non-convex quadratically constrained quadratic problem that finds optimal dosing schedules given observed schedules from clinical studies.



\paragraph{Liver transplantation.}

Several papers have formulated Markov decision process models to determine when a patient should accept a living-donor transplant~\citep{alagoz2004optimal, alagoz2007choosing, alagoz2007determining}. In~\cite{erkin2010eliciting}, the authors propose the inverse problem. By assuming a control-limit policy over ordered health states describing when a patient switches from waiting to accepting an offered liver, inverse optimization can determine patient preferences over the health states that make the observed policy optimal. Their paper applied formulation \eqref{eq:inverse_mdp} with a weighted 1-norm in the objective and polyhedral constraints on the cost vector to generate a linear inverse problem.



\subsection{Additional applications}\label{subsec:other_apps}


\begin{itemize}
    \item \textbf{Finance:} Inverse optimization can be used to learn risk preferences of investors by using data on investment strategies.  
    Consider the Markowitz mean-variance portfolio optimization problem
    \begin{align*}
    \max_{\bx} \quad & \bx^\top \bQ \bx + \kappa \br^\top \bx\\
    \st \quad & \bA \bx \geq \bb,
    \end{align*}
    where $\bQ$ represents the covariance of stock returns, $\br$ is the vector of mean returns, and $\kappa$ is the preferred risk tolerance of the portfolio. Here, the linear constraints $\bA\bx \geq \bb$ denote any additional restrictions that may be imposed on investments. If we assume that observed investment decisions approximately reflect solutions to the Markowitz model, and that we have observed data $\{(\bhx_i, \br_i, \bQ_i, \bA_i, \bb_i)\}_{i=1}^N$, then we can use inverse optimization to infer the value of $\kappa$. Risk preference learning in Markowitz models is studied in \cite{yu2021learning} using Inverse Online Learning with Minimum Distance loss (see Section \ref{subsubsec:IOL_inverseoracle}). The inverse model uses the optimality conditions
    \begin{align*}
        \Xopt(\kappa) = \left \{ \bx \; \bigg | \; \bA\bx \geq \bb, \; \by^\top(\bA\bx - \bb) = 0, \; \bQ\bx - \kappa \br - \bA^\top \by = \bzero, \; \by \geq \mathbf{0} \right \}.
    \end{align*}
    \citet{li2021inverse} establish a more general framework for risk preference inference by developing inverse models for general, convex risk functions. Learned preferences can be used by financial planners and new automated robo-advising systems to provide investment support that is tailored to individual tastes \citep{alsabah2019robo}.  They can also be used to assess additional properties of existing portfolios. For example, \citet{utz2014tri} uses a variant of the inverse Markowitz model to assess the level of environmental and social consideration that is given to socially-responsible mutual fund products when designed by different fund managers. 

    \item \textbf{Economics:} Inverse models have developed for multi-player Nash games, which are well-studied and appear in different application domains. In these games, each player $i \in \{1, \dots, N\} $ is assumed to simultaneously solve the following optimization problem, where $\bx^i$ denotes player $i$'s decision and $\bx^{-i}$ denotes the decisions of all other players:
    \begin{align*}
        \max_{\bx^i} \left \{ f(\bx^i, \bx^{-i}, \btheta^i) \; \bigg | \; \bx^i \in \mX^i \right \} \quad \forall i \in \{1, \dots, N\}.
    \end{align*}
    The objective function is typically assumed to be linear or quadratic in $\bx^i$ and $\bx^{-i}$. These games are often used to model decentralized markets where market participants compete by deciding on production quantity. Under mild assumptions, the equilibria of these multi-player games can be described using variational inequalities and KKT conditions. To infer the $\{\btheta_i\}_{i=1}^N$ terms, \citet{bertsimas2015data} consider a data-driven inverse VI approach while \citet{ratliff2014social} and \citet{risanger2020inverse} apply a data-driven inverse KKT approach. \citet{allen2021using} extends the inverse KKT approach to generalized Nash equilibria, where player decisions not only affect each other's objectives but also each other's feasible regions, for example, through a joint capacity constraint.
    
    \item \textbf{Biology:} Optimization models are used to model certain cellular processes. For example, linear optimization can be used to model the metabolism of a cell \citep{orth2010flux}. \citet{ref:zhao_cdc15} and  \citet{zhao2016mapping} apply data-driven inverse linear models with sub-optimality loss to infer objective functions that best explain and predict cellular metabolism behavior.

\end{itemize}

\section{Conclusion}\label{sec:conclusion}

This paper reviews the extensive literature on inverse optimization. 
We conclude with what we believe are two fruitful directions for future research. 

The first direction is to develop more efficient solution methods. As we have pointed out in this review, many inverse models do not scale well with data. For example, when decision data points are drawn from different feasible regions, each data point requires the introduction of a new set of variables and constraints into the inverse model. Future work could aim to improve the solvability of large-scale inverse models, for example through novel decomposition or approximation methods. 

The second direction is to develop inverse models that can better incorporate distributional information. The literature to date has little to no distributional assumptions about the input data, model parameters, or the degree of model mis-specification. This is in direct contrast with other models in statistics or econometrics, which rely heavily on distributional information. Incorporating such information may improve the estimation or prediction quality of the models.


\bibliographystyle{plainnat}
\bibliography{InverseBib}

\ECSwitch


\ECHead{Electronic Companion}


\section{Additional Details for Select Applications}

In this section we provide additional mathematical details of some applications that are discussed in the main body of the paper.

\subsection{Toll design for risk mitigation}\label{ECsubsec:toll_design}


Let $\mG(\mV, \mA)$ denote a network with nodes $\mV$ and arcs $\mA$. Assume the carrier's problem is a multi-commodity flow problem over $\mG$ minimizing the cost of transporting $S$ types of hazardous materials between a set of source-sink pairs. For each material type $s \in \{1, \dots, S\}$, let $c_a^s$ denote the cost of transporting the material $s$ on arc $a \in \set{A}$ and let $x_a^s \in \{0, 1\}$ be a binary variable that equals one if material $s$ is transported on arc $a$. Given resource and flow balance constraints in the form of $\set{X}$, the carrier's multi-commodity flow problem is
\begin{align*}
\min_{\bx} \; \left \{ \sum_{s = 1}^S \sum_{a \in \mA} c_{a}^s x_{a}^s \; \bigg | \; \bx \in \mX \right \}.
\end{align*}
We use $\Xopt(\bc)$ to denote the set of optimal solutions to this problem under transportation costs $\bc$. Now suppose the estimated risk posed by transporting material $s$ over arc $a$ is $\rho^s_{a}$. To find a set of arc tolls $\btheta \in \bTheta$ that makes the minimum-risk solution optimal for the carrier, we can solve
\begin{align*}
\min_{\bx, \btheta} \; \left \{ \sum_{s = 1}^S \, \sum_{a \in \mA} \rho_{a}^s x_a^s \; \bigg | \; \bx \in \Xopt(\bc + \btheta), \; \btheta \geq \mathbf{0} \right \}.
\end{align*}

To solve this problem, we can use the master-subproblem technique, described in Section \ref{sec:classic_IO_apps}, which decouples the computation of the lowest-risk solution and the corresponding tolls into two sequential linear programs. Specifically, we first solve the master problem 
\begin{align*}
\min_{\bx} \; \left \{ \sum_{s=1}^S \, \sum_{a \in\mA} \rho_{a}^s x_{a}^s \; \bigg | \; \bx \in \mX \right \},
\end{align*}
to yield a feasible flow $\bhx$ that minimizes the risk. Then we solve an inverse optimization subproblem to obtain the tolls for which $\bhx$ is now also optimal for the carrier, i.e.,
\begin{align*}
\min_{\btheta} \; \left \{ \sum_{s=1}^S \, \sum_{a\in\mA} \theta_{a}^s \hat x_{a}^s \; \bigg | \; \bhx \in \Xopt(\bc + \btheta), \; \btheta \geq \mathbf{0} \right \}.
\end{align*}
Note that this inverse problem computes the \emph{minimum} amount of tolls that must be collected from the carrier to induce the minimum-risk solution.

\subsection{Estimating traffic equilibrium models}\label{subEC:traffic_VI}

We first describe the traffic assignment model (i.e., the forward problem which produces equilibrium conditions), then provide details on inverse optimization framework used to estimate this model.

\paragraph{The forward problem.} Let $\mG = (\mV, \mA)$ define a road network consisting of a set of $m$ nodes $\mV$ and a set of $n$ directed arcs $\mA$, and let $\mP = \mV \times \mV$ denote the set of all node pairs in the network. Let $\bA \in \{0,1,-1\}^{m \times n}$ denote the node-arc incidence matrix. The value $d_{i,j}$ denotes the demand between each pair of nodes $(i,j) \in \mP$, and we define a corresponding vector $\bb^{i,j} \in \mathbb{R}^m$ where $\bb^{i,j}$ is a vector of all zeros except for elements $i$ and $j$ where $b^{i,j}_i = d_{i,j}$ and $b^{i,j}_j = - d_{i,j}$.  Let $\bx^{i,j} = (x^{i,j}_1, \ldots, x^{i,j}_{n})$ denote a vector of flows on each arc corresponding to demand between node $i$ and $j$. Finally, let $\mX$ be the set of aggregate flows satisfying all pairwise demand, i.e.,
\begin{align*} 
    \mX = \left\{ 
    \bx \; \Bigg| \ \bx = \sum_{(i,j) \in \mP} \bx^{i,j}, \; \bA \bx^{i,j} = \bb^{i,j}, \; \bx^{i,j} \geq \mathbf{0}, \ \ \forall (i,j) \in \mP \right\}.
\end{align*}

For a given flow $\bx \in \mX$, let $\bff(\bx) = (f_1(x_1), \ldots, f_{n}(x_n))$ be the vector of marginal travel time costs on each arc under $\bx$. In the transportation literature, these costs are commonly modeled as
\begin{align*}
f_a(x_a) = c_a \, g\left(\frac{x_a}{m_a}\right),
\end{align*}
where $c_a$ is a parameter measuring the uncongested travel time of arc $a$, $m_a$ is a parameter corresponding to the ``capacity" of an arc (e.g., the number of lanes and the speed limit) and $g(x)$ is a monotonically increasing function modeling congestion effects. Thus, arc costs are differentiated only by $c_a$ and $m_a$, where they increase when $c_a$ or $x_a$ increase, or when $m_a$ decreases. 

Given a demand matrix and known cost functions, different flow solutions can be computed using various transportation models which differ by the underlying assumptions made. A widely studied model is the traffic assignment problem \citep{dafermos1969traffic, patriksson2015traffic}, which produces solutions describing decentralized traffic flow, where every driver has complete information on costs and makes their routing decisions independent of decisions made by other drivers. This defines the forward model
\begin{align}\label{eq:TAP}
    \min_{\bx} \left\{ \sum_{a = 1}^n \int_0^{x_a} f_a(s) \; ds \; \Bigg | \; \bx \in \mX\right\}.
\end{align}
Since $\bff(\bx)$ represents the marginal travel time cost, it is the gradient of the objective function in the forward model.
This model is considered to be a good approximation for describing real-world traffic movement. However, for this model to be of practical use, $\bff(\bx)$, and more specifically the $g(.)$ function, must be known. The standard approach in transportation modeling is to assume a specific $g(.)$ function, a common choice being $g(\frac{x_a}{m_a}) = (1 + 1.15(\frac{x_a}{m_a})^4)$, as well as a specific value of $m_a$~\citep{chow2014nonlinear, bertsimas2015data}. However, assuming such functions, rather than deriving from data, can result in poor modeling and out-of-sample performance.

\paragraph{The inverse problem.}

Suppose we observe spatiotemporal data of transportation flows and demands over $T$ periods, denoted by $\{(\bhx_t,\mX_t)\}_{t=1}^T$ with $\bhx_t \in \mX_t$ for all $t$. In the inverse problem, we estimate the marginal travel time cost functions $\bff(\bx, \btheta)$ where $f_a(x_a, \btheta) = c_a g(\frac{x_a}{m_a}, \btheta)$ are now parametrized by $\btheta$. Following \citet{bertsimas2015data} and \citet{zhang2018price}, we model $g(\frac{x_a}{m_a}, \btheta)$ using polynomial kernels
\begin{align}\label{app:poly_kernel}
g\left(\frac{x_{a}}{m_a}, \btheta \right) = 1 + \theta_1 \left( \frac{x_a}{m_a}\right) + \ldots + \theta_n \left(\frac{x_a}{m_a}\right)^n.
\end{align}
In order to estimate $\btheta$ using the data set $\{(\bhx_t,\mX_t)\}_{t=1}^T$, \citet{bertsimas2015data} leverage an important property from the transportation literature, that when $\bff(\bx, \btheta)$ is strongly monotonic and continuously differentiable, the optimal traffic flows satisfy a Wardrop equilibrium~\citep{dafermos1969traffic, patriksson2015traffic}. That is, the unique optimal solution $\bx^*$ to $\eqref{eq:TAP}$ is also the unique solution to the following set of variational inequalities
\begin{align*}
\bff(\bx^*, \btheta)^\top (\bx - \bx^*) \geq 0, \quad \forall \bx \in \mX.
\end{align*}
The Wardrop equilibrium ensures that for every $(i,j) \in \mP$, and for any route between the pair with positive flow in $\bx^*$, the cost of traveling along that route is no greater than the cost of traveling along any other feasible routes between $i$ and $j$ \citep{patriksson2015traffic}. Practically, this implies that every driver acts selfishly and takes the lowest cost route. 

With this connection to variational inequalities, and with the observation that the forward problem is conic, we can employ the Inverse Variational Inequality problem (see Section \ref{sec:datadriven_vi}, and in particular Theorem \ref{thm:ivi_duality_solution}) for problem~\eqref{eq:TAP} to estimate the arc cost functions by solving the following inverse optimization problem
\begin{subequations}
\begin{align*}
    \min_{\bepsilon, \blambda, \btheta} \quad &  \norm{\bepsilon}_p + \kappa \sum_{i = 1}^n \delta_i \theta_i^2 \\ 
    \st \quad &  \sum_{a \in \mA} c_a \hat x_{a,t} \, g\left(\frac{\hat x_{a,t}}{m_a}, \btheta \right)  - \sum_{(i,j) \in \mP} (\bb^{ij}_{t})^\tpose \blambda^{ij}_{t} \: \leq \: \epsilon_t, \quad \forall \, t \in \{1, \ldots, T\},\\[0.02cm]
    & \bA_{a}^\tpose \blambda_{t}^{ij} - c_a \, g\left(\frac{\hat x_{a,t}}{m_a}, \btheta \right) \: \leq \: 0, \quad \forall \, a \in \mA, \; t \in \{1, \ldots, T\}, \\
    & \btheta \in \bTheta.
\end{align*}
\end{subequations}
The vector $\bA_a$ denotes the a-th row of matrix $\bA$, and the variables $\blambda_t^{ij} \in \mathbb{R}^{m}$ are duals associated with data point $t$ and node-pair $(i,j) \in \mP$. The parameters $\delta_i$ define a penalization term for polynomial kernels \citep[see][]{bertsimas2015data, zhang2018price}. The set $\bTheta$ includes constraints that impose monotonicity on $\bff(\cdot)$ over the observed flow ranges, either in the form of a general constraint $\btheta \geq \mathbf{0}$ or as a set of individual constraints on each arc in the form of $f_a(x_a) \geq f_a(\tilde{x}_a)$ for all observed  arc flows $x_a$ and $\tilde{x}_a$ where $x_a \geq \tilde{x}_a$. Note that because $g(\frac{x_a}{m_a}, \mathbf{0}) = 1$ in equation \eqref{app:poly_kernel}, we avoid inferring the ``trivial" objective vector $g(\cdot) = 0$ that would render all observed solutions equivalent and optimal. Finally, the regularization term $\kappa$ in the inverse problem can be tuned using cross-validation.

\citet{zhang2016price} and \citet{zhang2018price} show that the estimated cost functions can be used in a different model to measure network efficiency. Specifically, we can use $\{f_a(\cdot)\}_{a \in \mA}$ to solve a system-optimal routing model, which minimizes the total cumulative cost of travel for all drivers:
\begin{equation}\label{eq:SOP}
\min_{\bx} \left\{ \sum_{a \in \mA} x_a f_a(x_a) \; \Big | \; \bx \in \mX \right\}.
\end{equation}

Unlike solutions from model \eqref{eq:TAP}, solutions from model \eqref{eq:SOP} are Pareto efficient. The ratio of the total travel time between system-optimal and user-optimal solutions, known as the ``price of anarchy", provides a measure of the efficiency of the system. While the ratio has been studied extensively as a theoretical concept, the inverse optimization approach enable the authors to provide one of the first empirical estimates of this ratio. 

\subsection{Inferring constraints of market-clearing models}\label{subEC:market_clearing}

Consider a market with $J$ participants situated across $I$ different nodes. Let $v_j$ denote the total amount of electricity consumed ($v_j < 0$) or produced ($v_j > 0$) by participant $j \in \{1,\ldots,J\}$ who is situated at node $n(j)$. Let $S_j(v_j)$ describe an offer or bid function submitted by participant $j$, i.e., $S_j(v_j)$ (or $-S_j(v_j)$) denotes the price that a consumer (or producer) $j$ is willing to pay (or be paid) for $v_j$. We assume $S_j(v_j)$ is stepwise linear. Finally, let $x_i$ represent the total amount of electricity consumed or produced at node $i \in \{1, \ldots, I\}$. The market clearing model as defined in \cite{birge2017inverse} maximizes social welfare subject to constraints on $M$ different transmission links as well as constraints  $\bv \in \mV$ defining individual production and consumption, i.e.,
\begin{subequations}
\begin{align}
\max_{\bx, \bv} \quad &  \sum_{j=1}^J S_j(v_j)\\
\text{s.t.} \quad & \sum_{j : n(j) = i} v_j = x_i, \quad \forall i \in \{1, \dots, I\}, \label{cons:ece_nodes} \\
& \bA \bx \leq \bb \label{cons:ece_links}\\
& \bv \in \mV.
\end{align}
\end{subequations}
The matrix $\bA \in \mathbb{R}^{M \times I}$ and vector $\bb \in \mathbb{R}^{M}$ define the grid and the transmission capacities (based on DC power flow). The vector of optimal dual variables $\hat \bpi \in \mathbb{R}^I$ and $\hat \blambda \in \mathbb{R}^M$ corresponding to constraints \eqref{cons:ece_nodes} and \eqref{cons:ece_links} represent the prices of electricity at every node and the shadow prices of various capacity constraints, respectively; these are used to decide payments.

To recover the matrix $\bA$, \citet{birge2017inverse} assume a noise-free environment and use a publicly available data set $\{(\bhx_i, \hat \blambda_i, \hat \bpi_i)\}_{i=1}^N$ where $N \geq M$ to find a matrix $\bA$ and vectors $\bb_i$ that satisfy the following set of optimality conditions (\citet{birge2017inverse} include other constraints related to the physical transmission of electricity, which we omit here for simplicity): 
\begin{align*}
    \bpi_i = \hat\blambda_i \bA, \  \bA \bhx_i \leq \bb_i, \ \hat\blambda_i \odot (\bA\bhx_i - \bb_i) = \mathbf{0}, \quad \forall i \in \{1, \ldots, N\}.
\end{align*}
%


Finally, we note that the operations of electricity markets vary significantly across geographic regions,  
and these distinctions offer new research opportunities. For example, \citet{ruiz2013revealing} examine a market where the model constraints are published rather than the bids. The authors describe a duality-based inverse optimization model to infer stepwise bid functions of producers. 

\subsection{Clinical pathway concordance measurement} \label{sec:DPMEC}

Let $\mathcal{G}$ denote a graph with nodes $\mathcal{N}$ and arcs $\mathcal{A}$. The nodes describe the set of activities patients can undertake, including concordant activities like medical imaging and treatment, and discordant activities like emergency department visits and extra consultations. A walk through the graph accumulates costs along the arcs it traverses. Clinical pathways developed by experts are assumed to be shortest paths through $\mathcal{G}$. Implicit costs of arcs between activities can then be estimated with model~\eqref{model:IOclinpathway}, which identifies a single set of arc costs $\btheta$ that minimizes the aggregate sub-optimality of clinical pathways with respect to the difference between their costs and the shortest path cost.


While model~\eqref{model:IOclinpathway} considers the expert-defined clinical pathways as input, it does not consider actual patient data from either successful or unsuccessful clinical workflows. Consequently,~\citet{chan2022inverse} employ a second stage inverse problem to refine the cost vector $\btheta^*$ from model~\eqref{model:IOclinpathway} with real patient pathways.
Consider a set of patient-traversed pathways from patients who survived their cancer, $\hat\bx^s_1, \ldots, \hat\bx^s_S$ and a set of patient-traverse pathways from patients who died, $\hat\bx^d_1, \ldots, \hat\bx^d_D$. The refined problem penalizes the duality gaps with respect to patients who survived and encourages higher duality gaps for patients who died, while fixing the optimal duality gaps $\epsilon^{r*}_i$ from problem \eqref{model:IOclinpathway}. We can write this problem as
\begin{align} \label{model:patientref}
\begin{split}
\underset{\btheta, \blambda, \bepsilon^s, \bepsilon^d}{\text{min}} & \quad  \frac{D}{S}\sum_{j=1}^{S}\epsilon^s_j-\sum_{k=1}^{D}\epsilon^d_k \\
\textrm{s.t.} \; & \quad \bA^\tpose \blambda \le \btheta  \\
& \quad \btheta^\tpose \hat\bx^r_i=\bb^\tpose \blambda+\epsilon_i^{r*}, \quad \forall  i \in \{1, \ldots, N\} \\
& \quad \btheta^\tpose \hat\bx^s_j=\bb^\tpose \blambda+\epsilon^s_j, \quad  \forall j \in \{1, \ldots, S\} \\
& \quad \btheta^\tpose \hat\bx^d_k=\bb^\tpose \blambda+\epsilon^d_k, \quad  \forall k \in \{1, \ldots, D\} \\
& \quad \|\btheta\|_{\infty}= 1 \\
& \quad \bA \btheta = \bzero. 
\end{split}
\end{align}
%
The objective minimizes (maximizes) the aggregate sub-optimality with respect to the patients who survived (died). Because $\epsilon^{r*}_q$ is fixed, this model chooses among the optimal cost vectors from model~\eqref{model:IOclinpathway} to find one that maximizes the ``separation'' between the pathway costs of patients who survived from those who died. 

Finally, using the optimal cost vector $\btheta^*$ from problem~\eqref{model:patientref},~\citet{chan2022inverse} construct a concordance metric $\omega \in [0,1]$ as follows: 
\begin{align*}
    \omega(\hat\bx) = 1 - \frac{\btheta^{*\top}\hat\bx - \btheta^{*\top}\bx^*}{M(\hat\bx) - \btheta^{*\top}\hat\bx}.
\end{align*}
This metric measures the cost difference between a patient pathway $\hat\bx$ and a shortest path $\bx^*$, and normalizes it based on the cost difference between the longest walk with the same number steps as $\hat\bx$ (denoted $M(\hat\bx)$) and a shortest path. Using $\omega$, it becomes possible to rigorously score any patient pathway $\bhx$ against the clinical pathways. Using a real dataset of colon cancer patients, the authors establish a statistically significant association between concordance and mortality, even after adjusting for patient covariates, which supports the clinical meaningfulness of the metric.

An extension that has not been considered yet is the incorporate of cost of care into the concordance measurement. This can be done by modifying formulation~\eqref{model:patientref} to favor patient pathways that are lower cost (in terms of real dollar amounts of imaging, treatment, diagnostic tests, etc.), rather than separating patients by their survival outcome.

\subsection{Radiation therapy treatment planning}\label{sec:RTEC}

The optimization problem contains multiple objectives used to balance multiple (potentially dozens) conflicting objectives, such as escalating dose to the tumor while minimizing dose to the healthy organs. The overall objective is formed by taking a weighted combination of the various objectives $f(\bx) = \sum_{k=1}^K \theta_k f_k(\bx)$ where $\theta_k$ is the weight of the $k$-th objective $f_k(\bx)$ for $k \in \{1, \dots, K\}$. Assuming linear objectives (i.e., $f_j(\bx) = \bc^{j\top}\bx$), the objective function can be rewritten as $\btheta^\top \bC \bx$, where the $j$-th row of $\bC$ is $\bc^j$.

The traditional clinical procedure of designing treatments involves manually selecting the objective function weights, solving the optimization problem, evaluating the treatment using several quantitative and qualitative metrics, and then iterating, if needed. As an alternative, inverse optimization can infer appropriate weights from historical treatments~\citep{chan2014generalized}. We present a simplified version of this model below.

Let $\mathcal{B}$ be the sent of beamlets and $x_b$ be the intensity of beamlet $b$. Let $\mathcal{T}$ be the set of voxels (volumetric pixels) in the tumor and $\mathcal{V}$ be the set of all voxels. Let $\mathcal{O}_k$ be the set of healthy voxels corresponding to the $k$-th objective. Each objective $f_k(\bx)$ penalizes the total dose delivered above a threshold $\tau^k_v$ for each voxel $v$, i.e.,
\begin{align*}
    f_k(\bx) := \sum_{v \in \mathcal{O}_k}\max\left\{0, \sum_{b \in \mathcal{B}} D_{v,b}x_b - \tau^k_v\right\},
\end{align*}
where $D_{v,b}$ is the dose deposited to voxel $v$ by unit intensity of beamlet $b$. The complete formulation of the forward problem is
\begin{equation} \label{model:RTfwd}
\begin{alignedat}{1}
\underset{\bx}{\text{min}} & \quad \sum_{k=1}^K \theta_k \sum_{v \in \mathcal{O}_k}\max\left\{0, \sum_{b \in \mathcal{B}} D_{v,b}x_b - \tau^k_v\right\} \\
\textrm{s.t.} & \quad \sum_{b \in \mathcal{B}} D_{v,b}x_b \ge l_v, \quad \forall v \in \mathcal{T}, \\
& \quad \sum_{b \in \mathcal{B}} D_{v,b} x_b \le u_v, \quad \forall v \in \mathcal{V}, \\
& \quad \bx \in \mX,
\end{alignedat}
\end{equation}
where $l_v$ and $u_v$ denote lower and upper bound constraints on the tumor and healthy voxels, respectively, and $\mX$ describes a set of linear constraints on the intensity values including nonnegativity. 

Both the Absolute and Relative Sub-optimality loss functions (see Sections \ref{subsubsec:abs_suboptimality} and \ref{subsubsec:data-driven_IO_rdg}) have been used in the literature in the inverse formulation for the above forward problem.

\end{document}